\documentclass[a4paper,twoside,10pt]{article}
\usepackage[a4paper,left=3cm,right=3cm, top=3cm, bottom=3cm]{geometry}
\usepackage[utf8]{inputenc}
\usepackage{mathrsfs}
\usepackage{graphicx}
\usepackage{epstopdf}
\usepackage{amsmath}
\usepackage{amsthm}
\usepackage{appendix}
\usepackage{mathbbol}
\usepackage{amssymb}
\DeclareSymbolFontAlphabet{\amsmathbb}{AMSb}
\usepackage{tikz}
\usepackage{subfigure}
\usepackage{float}
\usepackage{cite}
\usepackage{color}
\usepackage{hyperref}
\usepackage[font=footnotesize,labelfont=bf]{caption}
\usepackage{cases}
\usepackage{subfigure}
\usepackage{tikz}
\usetikzlibrary{calc}
\usepackage{mathbbol}
\usetikzlibrary{arrows.meta,bending,positioning} 
\usepackage{cancel}

\newcommand{\eremk}{\hbox{}\hfill\rule{0.8ex}{0.8ex}}
\newcommand{\Vertiii}[1]{{\left\vert\kern-0.25ex\left\vert\kern-0.25ex\left\vert #1 
		\right\vert\kern-0.25ex\right\vert\kern-0.25ex\right\vert}}

\restylefloat{table}
\theoremstyle{plain}
\newtheorem{theorem}{Theorem}[section]

\newtheorem{lemma}[theorem]{Lemma}
\newtheorem{proposition}[theorem]{Proposition}
\newtheorem{remark}[theorem]{Remark}

\numberwithin{equation}{section}

\newcommand{\Norm}[2]{\left\Vert #1 \right\Vert_{#2}}
\newcommand{\SemiNorm}[2]{\left\vert #1 \right\vert_{#2}}
\newcommand{\jump}[1]{\left[\!\left[#1\right]\!\right]}
\newcommand{\average}[1]{\left\{\!\!\left\{#1\right\}\!\!\right\}}
\let\div\relax
\DeclareMathOperator{\div}{div}
\DeclareMathOperator{\divbf}{\bf div}
\DeclareMathOperator{\divdivbf}{\div\divbf}
\DeclareMathOperator{\divdivbfz}{\div\divbf^0}
\DeclareMathOperator{\divdivbfh}{\div\divbf_{\h}}
\DeclareMathOperator{\curl}{curl}
\DeclareMathOperator{\curlbf}{\bf curl}
\DeclareMathOperator{\curlbb}{\underline{\mathbf{curl}}}

\DeclareMathOperator{\sym}{sym}

\DeclareMathOperator*{\argmin}{arg\,min}

\newcommand{\p}{p}
\DeclareMathOperator{\h}{h}
\DeclareMathOperator{\uh}{u_{\h}}

\DeclareMathOperator{\R}{R}
\DeclareMathOperator{\Z}{Z}
\newcommand{\Rbb}{\mathbb R}
\newcommand{\RT}{\mathbf{RT}}
\newcommand{\RTbb}{\underline{\mathbf{RT}}}
\newcommand{\nubold}{\boldsymbol \nu}
\newcommand{\nuboldOmega}{\nubold_{\Omega}}
\newcommand{\taubold}{\boldsymbol \tau}
\newcommand{\tauboldOmega}{\taubold_{\Omega}}
\newcommand{\vbf}{\mathbf v}
\newcommand{\nuboldF}{\nubold_\F}
\DeclareMathOperator{\Nucalh}{\mathcal V_{\h}}
\newcommand{\omeganu}{\omega^\nu}
\DeclareMathOperator{\tauh}{\mathcal T_{\h}}
\DeclareMathOperator{\tautildeh}{\widetilde{\mathcal T}_{\h}}
\newcommand{\Ecal}{\mathcal E}
\let\L\relax
\DeclareMathOperator{\L}{L}
\let\H\relax
\DeclareMathOperator{\H}{H}
\newcommand{\Lbb}{\underline{\mathbf L}}
\DeclareMathOperator{\Lbbh}{\underline{\Lbf}_{\h}}
\newcommand{\xbf}{\mathbf x}

\newcommand{\EcalT}{\Ecal^\T}

\newcommand{\hF}{\h_\F}
\DeclareMathOperator{\symcurlbb}{sym\underline{\curlbf}}
\DeclareMathOperator{\curlsym}{\curlbf\sym}
\DeclareMathOperator{\curlhsym}{\curl_{\h}\sym}
\DeclareMathOperator{\Vh}{V_h}
\DeclareMathOperator{\vh}{v_h}
\DeclareMathOperator{\wh}{w_h}
\newcommand{\boldalpha}{\boldsymbol\alpha}
\newcommand{\qp}{q_\p}
\newcommand{\thetabf}{\boldsymbol \theta}
\DeclareMathOperator{\thetabfh}{\thetabf_{\h}}
\newcommand{\T}{T}
\newcommand{\Ttilde}{\widetilde\T}
\newcommand{\hT}{\h_\T}
\newcommand{\nuboldT}{\nubold_\T}
\newcommand{\tauboldT}{\taubold_\T}
\newcommand{\tauboldF}{\taubold_\F}
\DeclareMathOperator{\nablah}{\nabla_{\h}}

\newcommand{\phinu}{\varphi^\nu}
\newcommand{\homeganu}{\h_{\omeganu}}

\newcommand{\etaT}{\eta_\T}

\newcommand{\omegaT}{\omega^\T}
\newcommand{\ceff}{c_{\rm eff}}
\newcommand{\crel}{c_{\rm rel}}
\newcommand{\wbb}{\underline{\mathbf w}}
\DeclareMathOperator{\psibfh}{\boldsymbol \psi_{\h}}

\DeclareMathOperator{\Bh}{B_{\h}}

\newcommand{\bT}{b_{\T}}
\newcommand{\btildeF}{\widetilde b_{\F}}
\newcommand{\btildeT}{\widetilde b_{\T}}

\DeclareMathOperator{\f}{f}
\DeclareMathOperator{\fh}{f_{\h}}

\newcommand{\rhoT}{\rho_\T}

\newcommand{\Fcal}{\mathcal F}
\newcommand{\Ffrak}{\mathfrak F}
\newcommand{\FfrakDG}{\Ffrak_{\DG}}
\newcommand{\FfrakHbbh}{\Ffrak_{\Hbbh}}
\DeclareMathOperator{\Fcaln}{\Fcal_{\h}}
\DeclareMathOperator{\FcalnB}{\Fcal_{\h}^B}
\DeclareMathOperator{\FcalnI}{\Fcal_{\h}^I}
\newcommand{\F}{F}
\newcommand{\Ftilde}{\widetilde\F}
\newcommand{\Tplus}{\T^+}
\newcommand{\Tminus}{\T^-}
\DeclareMathOperator{\Dtwo}{\underline{\mathbf D}^2}
\DeclareMathOperator{\Dtwoh}{\underline{\mathbf D}^2_{\h}}
\newcommand{\csigma}{c_\sigma}
\newcommand{\ctau}{c_\tau}
\newcommand{\cLcalh}{c_{\Lbbh}}
\newcommand{\DG}{\rm DG}
\DeclareMathOperator{\Hbbh}{\underline{\mathbf H}_{\h}}

\newcommand{\Hbbhuh}{\Hbbh(\uh)}

\newcommand{\Sbb}{\underline{\mathbf S}}
\newcommand{\Tbb}{\underline{\mathbf T}}
\newcommand{\Hbf}{\mathbf H}
\newcommand{\Hbb}{\underline{\mathbf H}}
\newcommand{\Lbf}{\mathbf L}
\newcommand{\zerobf}{\mathbf 0}

\newcommand{\thetabfI}{\thetabf_{I}}
\newcommand{\hD}{\h_D}
\newcommand{\omegaF}{\omega_\F}
\newcommand{\omegatildeF}{\widetilde\omega_\F}

\DeclareMathOperator{\VHoh}{V^{\H^1}_{\h}}

\DeclareMathOperator{\VbfHoh}{\mathbf V^{\H^1}_{\h}}
\DeclareMathOperator{\VHth}{V^{\H^2}_{\h}}

\newcommand{\vI}{v_I}
\newcommand{\FcalT}{\mathcal F^\T}
\newcommand{\FcalOmega}{\mathcal F^\Omega}
\newcommand{\Hnzoh}{\H_{n,0}^\frac12}
\newcommand{\Hezth}{\H_{e,0}^\frac32}
\newcommand{\Hnmoh}{\H_{n}^{-\frac12}}
\newcommand{\Hemth}{\H_{e}^{-\frac32}}
\newcommand{\Hoh}{\H^{\frac12}}
\newcommand{\Hohzz}{\Hoh_{0}}
\newcommand{\vtilde}{\widetilde v}
\newcommand{\Ccal}{\mathcal C}
\newcommand{\VcalF}{\mathcal V^\F}
\DeclareMathOperator{\sign}{sign}
\newcommand{\partialnbfOmega}{\partial_{\nuboldOmega}}
\newcommand{\partialtbfOmega}{\partial_{\tauboldOmega}}
\newcommand{\partialtwotbfOmegatwo}{\partial^2_{\tauboldOmega^2}}

\newcommand{\partialtbfT}{\partial_{\tauboldT}}
\newcommand{\partialtbfF}{\partial_{\tauboldF}}
\newcommand{\wT}{w_\T}
\newcommand{\wbfT}{\mathbf w_\T}
\newcommand{\wF}{w_\F}

\newcommand{\wbfF}{\mathbf w_\F}
\newcommand{\bF}{b_\F}
\newcommand{\bstarF}{b_\F^*}
\newcommand{\To}{\T_1}
\newcommand{\Ttw}{\T_2}
\newcommand{\Tj}{\T_j}
\newcommand{\Abb}{\underline{\mathbf A}}

\renewcommand{\Bbb}{\underline{\mathbf B}}
\DeclareMathOperator{\Abbh}{\Abb_{\h}}

\newcommand{\gimelT}{\gimel_\T}

\newcommand{\psibf}{\boldsymbol \psi}
\newcommand{\Obb}{\underline{\mathbf O}}
\newcommand{\taubf}{\boldsymbol \tau}
\newcommand{\taubfOmega}{\taubf_\Omega}
\newcommand{\Sigmaj}{\Sigma_j}
\newcommand{\xibold}{\boldsymbol{\xi}}
\DeclareMathOperator{\divz}{\div^0}
\newcommand{\sigmabold}{\boldsymbol \sigma}
\newcommand{\xiboldtilde}{\widetilde{\xibold}}
\newcommand{\sigmaboldj}{\sigmabold_j}
\newcommand{\sigmaboldi}{\sigmabold_i}
\newcommand{\sigmaboldh}{\sigmabold_{\h}}
\newcommand{\Bbf}{\mathbf B}

\newcommand{\Sbbh}{\underline{\mathbf S}_{\h}}
\newcommand{\taun}{\mathcal T_{\h}}
\newcommand{\Mbbh}{\underline{\mathbf M}_{\h}}
\newcommand{\Qbbh}{\underline{\mathbf Q}_{\h}}
\newcommand{\CNL}{C}
\newcommand{\psibftildenu}{\widetilde{\psibf}^\nu}
\newcommand{\psibftilde}{\widetilde{\psibf}}

\newcommand{\xibolddelta}{\xibold_\delta}

\newcommand{\xiboldn}{\xibold_n}

\newcommand{\sigmaboldhn}{\sigmabold_{\h,n}}

\newcommand{\thetan}{\theta_n}
\newcommand{\rhodelta}{\rho_\delta}
\newcommand{\psij}{\psi_j}
\newcommand{\Hbfstardivz}{\Hbf^{*}(\divz,\Omega)}
\newcommand{\Hbfostar}{[\Hbf^1(\Omega)]^{*}}
\DeclareMathOperator{\RMbf}{\bf{RM}}
\DeclareMathOperator{\RMbfperp}{\RMbf^\perp}
\DeclareMathOperator{\psibfRMbf}{\psibf_{\RMbf}}
\DeclareMathOperator{\psibfRMbfperp}{\psibf_{\RMbfperp}}
\DeclareMathOperator{\nablaboldS}{\underline{\boldsymbol\nabla}^S}

\DeclareMathOperator{\PibfRMbfperpomeganu}{\boldsymbol\Pi_{\RMbfperp}^{\omeganu}}
\newcommand{\phibf}{\boldsymbol\varphi}

\author{\footnotesize{Th\'eophile Chaumont-Frelet\thanks{
Inria Univ. Lille and Laboratoire Paul Painlev\'e, 59655 Villeneuve-d'Ascq, France,
{\tt theophile.chaumont@inria.fr}},
Joscha Gedicke\thanks{Institut f\"ur Numerische Simulation, Universit\"at Bonn, 53115 Bonn,
{\tt gedicke@ins.uni-bonn.de}},
Lorenzo Mascotto\thanks{Dipartimento di Matematica e Applicazioni, Universit\`a di Milano-Bicocca, 20125 Milan, Italy;
IMATI-CNR, 27100, Pavia, Italy;
Fakult\"at f\"ur Mathematik, Universit\"at Wien, 1090 Vienna, Austria,
{\tt lorenzo.mascotto@unimib.it}}}}{}
\date{}

\date{}
\title{\footnotesize A generalized Hessian-based error estimator
for an IPDG formulation of the biharmonic problem
in two dimensions}

\begin{document}

\maketitle
\begin{abstract}
\noindent
\footnotesize We consider a two dimensional biharmonic problem
and its discretization by means of a symmetric interior penalty
discontinuous Galerkin method.
A novel split of an error measure based on a generalized Hessian
into two terms measuring the conformity
and nonconformity of the scheme
is proven.
This splitting is the departing point for the design of a new
error estimator,
which is provably reliable and efficient for polynomial degree larger than~$3$,
and does not involve any DG stabilization.
Such an error estimator can be bounded from above
by the standard DG residual error estimator.
Numerical results assess the theoretical predictions,
including the efficiency of the proposed estimator,
for all polynomial degrees larger than or equal to~$2$.
\medskip

\medskip
		
\end{abstract}

\section{Introduction} \label{section:introduction}

\paragraph*{State-of-the-art.}
Residual-type error estimators for discontinuous Galerkin (DG)
discretizations of elliptic problems are by now
a fairly well-understood topic. In~2003,
three works involving second order elliptic problems were published:
Rivi\`ere and Wheeler~\cite{Riviere-Wheeler:2003} investigated different
DG schemes and showed a posteriori error bounds
in the $\L^2$ norm;
Becker, Hansbo, and Larson~\cite{Becker-Hansbo-Larson:2003}
proved a posteriori error estimates in the energy norm
and, in the derivation of the upper bound,
they used a Helmholtz decomposition
to estimate the nonconforming part of the error,
cf. \cite{Dari-Duran-Padra-Vampa:1996};
Karakashian and Pascal discussed analogous estimates
employing different tools,
i.e., averaging operators mapping piecewise polynomials
into globally continuous piecewise polynomials.
DG residual error estimators typically contain
terms that also appear in residual estimators
for conforming methods
(elemental residuals, normal jumps of the gradients),
plus other jump terms that control the nonconformity of the discrete solution.

Such a structure of the error estimator
essentially extends to the case of
DG discretizations of fourth order problems.
Brenner, Gudi, and Sung~\cite{Brenner-Gudi-Sung:2010} exhibited a posteriori error estimators
for a lowest order $\mathcal C^0$ IPDG discretization
of the biharmonic problem in Hessian-Hessian formulation;
a general order IPDG discretization of the
same problem in Laplacian-Laplacian formulation
was investigated by Georgoulis, Houston,
and Virtanen in~\cite{Georgoulis-Houston-Virtanen:2011}.
In both references, two dimensional problems
were considered; essential tools
in the analysis were arithmetic averaging operators,
mapping continuous piecewise polynomials
and piecewise polynomials, respectively,
into $\mathcal C^1$ finite element spaces.
The extension to the three dimensional case
and polynomial degree explicit estimates
was tackled in~\cite{Dong-Mascotto-Sutton:2021},
where the averaging operator
was replaced by an elliptic reconstruction operator
into $\H^2$ functions, which are not necessarily
piecewise polynomials.
In all cases, the proposed error estimators
contained, in addition to terms that one would
expect from $\H^2$-conforming discretizations
(elemental residuals, jumps along certain directions
of broken second and third order derivatives),
also terms related to the nonconformity of the scheme
(jumps and normal jumps of the gradients).

Such stabilization terms are delicate for two main reasons.
First, they involve stability constants
depending on the degree~$p$ of the scheme;
this may influence the $p$ ``efficiency'' of the error estimators~\cite{Dong-Mascotto-Sutton:2021}.
Second, the choice of the stabilization parameters
affects the proof of the quasi-optimality of adaptive schemes
driven by that error estimator.
The works~\cite{Karakashian-Pascal:2007}
and~\cite{Bonito-Nochetto:2010}
discussed that the stabilization parameter
for DG discretizations of second order elliptic problems
must be even larger than what is needed
for the proof of the well-posedness
of the scheme, in order to show the convergence
and the quasi-optimality of those adaptive algorithms.
Effectively, it was proven in~\cite{Kreuzer-Georgoulis:2018}
for the second order case
and~\cite{Dominicus-Gaspoz-Kreuzer:2024}
for the fourth order case discretized with
a $\mathcal C^0$-IPDG scheme,
that this additional restriction is not really necessary
for the proof of the convergence of the scheme;
yet, it remains open the question whether
that condition is necessary or not for the proof
of quasi-optimality.

\paragraph*{Contents and contributions.}
We consider the biharmonic problem in two dimensions
with clamped boundary conditions
and its discretization by means of
a symmetric interior penalty method,
based on piecewise polynomials of general order larger than
or equal to~$2$.
We construct
a generalized Hessian~$\Hbbh(\uh)$
of the discrete solution~$\uh$ to the method,
which consists of the broken Hessian
and the fourth order DG lifting of~$\uh$
in the sense of display~\eqref{lifting} below.
With the notation of Section~\ref{section:continuous} below,
given $u$ the solution to the continuous problem,
for an error measure of the form
\[
\Norm{\Dtwo u-\Hbbh(\uh)}{\Omega},
\]
which does not depend on the stabilization
parameters of the scheme unless through $\uh$,
we construct an error estimator computable
from $\Hbbh(\uh)$ and the data of the problem only,
which is provably reliable and efficient
for the above error measure for all polynomial degrees
larger than~$2$; numerical results also show
that reliability and efficiency are valid for piecewise
quadratic polynomials as well.
This estimator does not contain explicitly
the stabilization terms that typically
appear in standard residual a posteriori error bounds.

\paragraph*{Crucial tools in the analysis.}
Amongst others, we pinpoint the following
results that are needed to derive
the a posteriori error bounds:
\begin{itemize}
    \item trace operators
    (see Proposition~\ref{proposition:trace-divdiv})
    and a Green identity
    (see display~\eqref{Green-divdiv})
    for the space of $\Hbb(\divdivbf)$ symmetric tensors;
    \item in order to handle domains with nontrivial topologies,
    we construct a decomposition of~$\Hbb(\divdivbf)$
    symmetric tensors into the sum of the~$\symcurlbb$ of certain vector fields
    also appearing for trivial topologies
    and Raviart-Thomas tensors accounting for the
    possibly nontrivial cohomology;
    (see Lemma~\ref{lemma:special-split}
    and Appendix~\ref{appendix:nontrivial-topologies});
    \item $\mathcal C^1$ quasi-interpolators mapping DG functions into globally $\H^2$ piecewise polynomials,
    which preserves values at the vertices
    (see display~\eqref{interpolation:C1}).
\end{itemize}

\paragraph*{Comments on the novel error estimator.}
The current setting has some differences
compared to that of standard residual error
estimators as in~\cite{Brenner-Gudi-Sung:2010, Georgoulis-Houston-Virtanen:2011, Dong-Mascotto-Sutton:2021}.
First, we consider an error measure~\eqref{T2-as-sup}
and design an error estimator~\eqref{global-error-estimator},
which do not involve stabilization terms explicitly.
This may lead in the future to different and possibly easier
proofs of convergence of related adaptive algorithms.
The new error estimator is reliable and efficient,
and can be estimated from above by the
standard residual one
(see Section~\ref{subsection:comparison-standard-ee}).
The derivation of the error measure and estimator
follows ideas from~\cite{ChaumontFrelet:2025}
with significant further elaboration related to the fourth order
nature of the problem.

\paragraph*{List of the main results.}
For the reader's convenience we spell out
a list of the main results of this work:
\begin{itemize}
    \item a novel split~\eqref{rewriting-in-terms-of-max}
    of an error measure involving the generalized Hessian
    in~\eqref{generalized:Hessian}
    into two terms that measure the conforming and nonconforming
    parts;
    \item the two essential properties listed in~\eqref{properties-GH}
    and proven in Section~\ref{subsection:properties:GH}
    that a generalized Hessian should satisfy
    in order to construct a reliable and efficient error estimator;
    \item the a posteriori error estimates in Theorem~\ref{theorem:apos-bounds}
    for the global error estimator in~\eqref{error-estimator};
    \item the comparison with the standard DG residual
    error estimator in Section~\ref{subsection:comparison-standard-ee}.
\end{itemize}

\paragraph*{Outline of the paper.}
We introduce some Sobolev spaces,
the $\Hbb(\divdivbf)$ complex,
and the continuous problem
in Section~\ref{section:continuous}.
The interior penalty method is presented in Section~\ref{section:method},
including a rewriting based on generalized Hessians.
Section~\ref{section:setting} is devoted
to derive an error measure, which can be split into
a combination of a conforming and a nonconforming error measures.
An error estimator and a posteriori error bounds
for those error measures
are detailed in Section~\ref{section:apos-bounds}.
Numerical results are presented in Section~\ref{section:nr}
and some conclusions are drawn in Section~\ref{section:conclusions}.
Appendices~\ref{appendix:nontrivial-topologies}
and~\ref{appendix:tools-nonconforming} are concerned with
(\emph{i}) the construction on domains with nontrivial topologies
of an explicit split of $\Hbb(\divdivbf)$ symmetric tensors
into the sum of the $\symcurlbb$ of vector fields
and certain Raviart-Thomas tensors;
(\emph{ii}) a Korn-type inequality for the $\symcurlbb$ operator.

\section{The continuous setting} \label{section:continuous}
\paragraph*{Differential operators.}
We denote the gradient, divergence, vector divergence,
Laplacian, bilaplacian,
vector curl, tensor curl,
symmetric tensor curl, and Hessian operators
by~$\nabla$, $\div$, $\divbf$, $\Delta$, $\Delta^2$,
$\curlbf$, $\curlbb$,
$\symcurlbb$,
and~$\Dtwo$.

\paragraph*{Admissible domains.}
In what follows,
\begin{equation} \label{Omega}
\Omega \text{ is a Lipschitz polygonal domain,
with possibly nontrivial topology.}
\end{equation}
The boundary of~$\Omega$ is given by the union of $N+1$ connected
components $\{\Sigmaj\}_{j=0}^N$, $\Sigma_0$ being the boundary
of the only unbounded connected component of $\Rbb^d \setminus \Omega$.

\paragraph*{Lebesgue and basic Sobolev spaces.}
In what follows, $D$ is a Lipschitz domain
with diameter~$\hD$ and boundary~$\partial D$.
The Lebesgue space of measurable functions with finite
$\Norm{v}{\L^2(D)}
= \Norm{v}{D}
:= (\int_D \vert v \vert ^2 )^\frac12$ norm
is denoted by~$\L^2(D)$;
this norm is induced by the inner product
$(u,v)_{D} := \int_D u \ v $.

Given a positive integer~$k$, $\H^{k}(D)$ denotes the space
of $\L^2(D)$ functions with distributional derivatives $D^{\boldalpha}$
of order~$k$ in $\L^2(D)$.
We introduce the seminorms and norms
\[
\SemiNorm{v}{k,D}
:= \big(\sum_{\vert \boldalpha \vert=k} 
        \Norm{D^{\boldalpha} v}{D}^2\big)^{\frac12},
\qquad\qquad
\Norm{v}{k,D}
:= \Big( \sum_{\ell=0}^k \big(\hD^{k-\ell}
            \SemiNorm{v}{\ell,D}\big)^2 \Big)^\frac12.
\]
Noninteger order Sobolev spaces are constructed via interpolation theory~\cite{Brenner-Gudi-Sung:2010}.
For vector and tensor Lebesgue and Sobolev spaces,
we replace the symbols~$\H$ and~$\L$
by $\Hbf$ and~$\Hbb$, and~$\Lbf$ and~$\Lbb$.
In the tensor case, additional subscripts~$\Sbb$ and~$\Tbb$
mean that we consider symmetric and trace free tensors.
As a matter of notation, we shall employ
standard, boldface, and underlined boldface
letters to denote scalars, vector fields, and tensors, respectively.

A trace theorem holds true for Sobolev spaces \cite[Chapter~3]{Ern-Guermond:2021}:
there exists a continuous operator mapping $\H^s(D)$ to~$\L^2(\partial D)$ for $s$ larger than $1/2$.
For $s$ in $(1/2, 3/2)$, we denote the image of~$\H^s(D)$
through the trace operator by~$\H^{s-\frac12}(\partial D)$.
For positive~$s$, $\H^s_0(D)$ is the closure of~$\mathcal C^\infty_0(D)$
with respect to the~$\H^s(D)$ norm.
The spaces $\H^1_0(D)$ and $\H^2_0(D)$ coincide
with the subspaces of~$\H^1(D)$ and~$\H^2(D)$
of functions with zero trace over~$\partial D$,
and zero trace and trace of the normal component of the gradient
over $\partial D$, respectively.

We define the Sobolev space $\H^{-1}(\Omega)$ by duality:
\[
\H^{-1}(D):= [\H^1_0(D)]^*,
\qquad\qquad\qquad
\Norm{w}{[\Hbf^1(D)]^*} :=
\sup_{v \in \H^1_0(D)} \frac{\langle w, v \rangle}{\Norm{v}{1,D}}.
\]

\paragraph*{Some Sobolev boundary spaces.}
For a given Lipschitz polygonal domain~$\Omega$
as in~\eqref{Omega},
we define two spaces on its boundary~$\partial\Omega$,
which will be useful to define trace operators
for a certain Sobolev space in~\eqref{Hdivdiv-spaces} below.
The tangent and (outward) normal vectors of $\partial\Omega$
are~$\tauboldOmega$ and~$\nuboldOmega$.

Given the set~$\Fcal^\Omega:=\{\F\}$ of facets of~$\partial\Omega$, let
\[
\Hohzz(\F)
:=
\{ v \in \H^{\frac12}(\F) \mid
    (\vtilde- v)_{|\F}=0 \text{ where } 
    \vtilde \in \Hoh(\partial\Omega),\    
    \vtilde_{\partial\Omega \setminus \F} =0 \}.
\]
In what follows, for any~$\F$ in~$\Fcal^\Omega$,
we shall also need the set $\VcalF$ of vertices of~$\F$.
We introduce a boundary space
as in \cite[Section~$2.3$]{Chen-Huang:2021}
\[
\Hnzoh(\partial\Omega)
:= \{ \partialnbfOmega v_{|\partial\Omega}
    \mid v \in \H^2(\Omega) \cap \H^1_0(\Omega) \}
 = \{ g \in \L^2(\partial\Omega)
        \mid g_{|\F} \in \Hohzz(\F) 
        \quad \forall\F\in\Fcal^\Omega \}
\]
and endow it with the norm
\[
\Norm{g}{\Hnzoh(\partial\Omega)}
:= \inf_{v \in \H^2(\Omega) \cap \H^1_0(\Omega),\ \partialnbfOmega v_{|\partial\Omega} = g} 
    \Norm{v}{2,\Omega}.
\]
We also introduce the space
\[
\Hezth(\partial\Omega)
:=
\{ v_{|\partial\Omega} \mid
    v\in \H^2(\Omega),\
    \partialnbfOmega v_{|\partial\Omega}=0,\quad
    v(\nu)=0\;\; \forall \nu \text{ vertices of } \Omega\}
\]
and endow it with the norm
\[
\Norm{g}{\Hezth(\partial\Omega)}
:= \inf_{\stackrel{v \in \H^2(\Omega) ,\
    \partialnbfOmega v_{|\partial \Omega} = 0,\
    v_{|\partial \Omega}=g}
    {v(\nu)=0\; \forall \nu \text{ vertices of } \Omega}}
    \Norm{v}{2,\Omega}.
\]
Finally, we consider the spaces
\[
\Hnmoh(\partial\Omega)
:= [\Hnzoh(\partial\Omega)]',
\qquad\qquad\qquad\qquad
\Hemth(\partial\Omega)
:= [\Hezth(\partial\Omega)]',
\]
which we can endow with their dual norms.

\paragraph*{The $\Hbb(\divdivbf,\Omega)$ space.}
Introduce the lowest order Raviart-Thomas space~$\RT_0(\Omega)$,
i.e., the span of $(1,0)$, $(0,1)$, and $(x,y)$,
and the spaces
\begin{equation} \label{Hdivdiv-spaces}
\begin{split}
\Hbb(\divdivbf,\Omega)
&:= \{ \Abb \in \Lbb^2(\Omega)
        \mid \divdivbf \Abb \in \L^2(\Omega)\}, \\
\Hbb(\divdivbf^{\zerobf},\Omega)
& := \{ \Abb \in \Lbb^2(\Omega)
        \mid \divdivbf \Abb = 0\} ,
\end{split}
\end{equation}
which we endow with the norm
\[
\Norm{\cdot}{\Hbb(\divdivbf,\Omega)}
:= ( \text{diam}(\Omega)^{-4} \Norm{\cdot}{\Lbb^2(\Omega)}^2
    + \Norm{\divdivbf \cdot}{\L^2(\Omega)}^2)^\frac12 .
\]
We define~$\Hbb_{\Sbb}(\divdivbf,\Omega)$ as the
subspace of symmetric tensors in~$\Hbb(\divdivbf,\Omega)$.
In what follows,
\[
\partialtbfOmega \cdot
\text{ and }
\partialnbfOmega \cdot
\text{ denote the tangential and normal derivatives along } \partial\Omega.
\]
A trace theorem for $\Hbb_{\Sbb}(\divdivbf)$ spaces is valid;
see \cite[Lemmas~$2.3$ and~$2.4$]{Chen-Huang:2021}
and~\cite{Fuhrer-Heuer-Niemi:2019, Fuhrer-Haberl-Heuer:2021}.

\begin{proposition} \label{proposition:trace-divdiv}
For any~$\Bbb$ in $\Hbb_{\Sbb}(\divdivbf,\Omega)$,
there exists a positive constant~$C$ such that
\small\[
\begin{split}
& \Norm{\nuboldOmega^T \Bbb \nuboldOmega}{\Hnmoh(\partial\Omega)}
\le C \Norm{\Bbb}{\Hbb(\divdivbf,\Omega)},
\qquad \Norm{\partialtbfOmega(\tauboldOmega^T \Bbb \tauboldOmega)
+ \nuboldOmega^T \divbf \Bbb}{\Hemth(\partial\Omega)}
\le C \Norm{\Bbb}{\Hbb(\divdivbf,\Omega)}.
\end{split}
\]\normalsize
The two trace operators above admit continuous right-inverses.
\end{proposition}
The proof of Proposition~\ref{proposition:trace-divdiv}
is strictly related to the validity of
the following Green's identity \cite[Lemma~2.1]{Chen-Huang:2021}:
for all~$v$ in $\H^2(\Omega)$ and all~$\Bbb$
in the space of symmetric tensors in $\Ccal^2(\Omega)$,
\begin{equation} \label{Green-divdiv}
\begin{split}
(\divdivbf \Bbb, v)_{\Omega}
& = (\Bbb,\Dtwo v)_{\Omega}
   - \sum_{\F \in \FcalOmega}
      \sum_{\nu \in \VcalF}
      \sign_{\F,\nu}\ (\tauboldOmega^T \Bbb \nuboldOmega)(\nu) \ v(\nu) \\
& \quad - \sum_{\F\in\FcalOmega} \Big[
        (\nuboldOmega^T\Bbb\nuboldOmega,\partialnbfOmega v)_{\F}
        - \big( \partialtbfOmega (\tauboldOmega^T\Bbb\nuboldOmega)
        + \nuboldOmega^T \divbf\Bbb,v)_{\F} \big) \Big],
\end{split}
\end{equation}
where, given an orientation of~$\F$
giving a start and an end endpoints,
\[
\sign_{\F,\nu}
:= \begin{cases}
    1  & \text{if~$\nu$ is the end endpoint of~$\F$}\\
    -1 & \text{if~$\nu$ is the start endpoint of~$\F$}
\end{cases}
\qquad\qquad\qquad
\forall \F \in\FcalOmega, \nu \in \VcalF.
\]
The last two inner products in~\eqref{Green-divdiv}
are to be interpreted as suitable duality pairings
for tensors~$\Bbb$ in $\Hbb_{\Sbb}(\divdivbf,\Omega)$:
\[
\begin{split}
\sum_{\F\in\FcalOmega}
(\nuboldOmega^T\Bbb\nuboldOmega,\partialnbfOmega v)_{\F}
& = _{-\frac12,\partial\Omega} \langle \nuboldOmega^T\Bbb\nuboldOmega, 
    \partialnbfOmega v \rangle _{\frac12,\partial\Omega},\\
\sum_{\F\in\FcalOmega} 
\big( \partialtbfOmega (\tauboldOmega^T\Bbb\nuboldOmega)
        + \nuboldOmega^T \divbf\Bbb,v)_{\F}
& = _{-\frac32,\partial\Omega} \langle \partialtbfOmega (\tauboldOmega^T\Bbb\nuboldOmega)
        + \nuboldOmega^T \divbf\Bbb,v
        \rangle _{\frac32,\partial\Omega}.
\end{split}
\]
The bounds in Proposition~\ref{proposition:trace-divdiv}
follow from~\eqref{Green-divdiv}, picking first~$v$
in~$\Hnzoh(\partial\Omega)$ and then in~$\Hezth(\partial\Omega)$.

\paragraph*{The model problem.}
Given~$f$ be in $\L^{2}(\Omega)$,
we are interested in finding~$u:\Omega \to \R$ such that
\begin{equation} \label{problem:strong}
\begin{cases}
\Delta^2 u = \f                             & \text{in } \Omega \\
u = 0, \qquad \nuboldOmega^T \nabla u = 0  & \text{on } \partial \Omega.
\end{cases}
\end{equation}
Introduce
\[
V := \H^2_0(\Omega),
\qquad\qquad\qquad\qquad
B(u,v) := (\Dtwo u, \Dtwo v)_{\Omega}
\qquad \forall u,v \in V.
\]
Noting that~$\Delta^2 = \div\divbf(\Dtwo)$,
a weak formulation of~\eqref{problem:strong} reads
\begin{equation} \label{problem:weak}
\begin{cases}
\text{find } u \in V \text{ such that}\\
B(u,v) = (\f , v )_{\Omega} \qquad\qquad \forall v \in V.
\end{cases}
\end{equation}
Problem~\eqref{problem:weak} is well-posed;
see, e.g., \cite[Section 5.9]{Brenner-Scott:2008}.

\section{The method and a generalized Hessian} \label{section:method}
\paragraph*{Meshes.}
We consider a simplicial mesh~$\tauh$.
Given~$\hT$ and~$\rhoT$ the diameter and the maximal inradius
of each element~$\T$ in~$\tauh$,
we introduce the positive shape-regularity parameter~$\gamma_\T$
\begin{equation} \label{def:gamma}
    \frac{\hT}{\rhoT} \le \gamma_T.
\end{equation}
In what follows, the constants appearing in the inequalities
are allowed to depend on $\gamma:=\max_{\T\in\tauh} \gamma_\T$;
on occasions, for given positive constants~$a$ and~$b$,
we shall write $a \lesssim b$ if there exists a positive constant~$C$
only depending on~$\gamma$.

With each~$\tauh$, we associate its set of total~$\Fcaln$,
internal~$\FcalnI$, and boundary~$\FcalnB$ facets.
The diameter of a facet~$\F$ is~$\hF$.
For a given mesh with given set of vertices
$\{\mathfrak v\}$, $\h$ denotes a function defined
on~$\Omega \setminus \{\mathfrak v\}$
such that it equals~$\hT$ and~$\hF$
in the interior of each element~$\T$
and in the interior of each facet~$\F$, respectively.
The set of facets of an element~$\T$ is $\FcalT$.
Given~$\T$ and~$\F$ in~$\tauh$ and~$\Fcaln$,
we define the patches
\small\begin{equation} \label{patches}
\omegaT := \cup \{ \Ttilde \in \tauh \mid 
    \Ttilde \cap \T \ne \emptyset \},
\qquad
\omegaF := \cup \{ \omegaT \mid
    \T\in\tauh,\; \F \in \FcalT \},
\qquad
\omegatildeF := \cup \{ \T \mid \F \in \FcalT \}.
\end{equation}\normalsize
For positive~$s$, we introduce the broken Sobolev space
\[
\H^s(\tauh,\Omega):= \{ v \in \L^2(\Omega) \mid v_{\T} \in \H^s(\T)
                    \quad \forall \T \in \tauh \}.
\]
For each facet~$\F$, we define the jump $\jump{\cdot}_\F$
and average $\average{\cdot}_{\F}$
operators on $\L^2(\F)$ as follows:
for a boundary facet~$\F$ and~$v$ in $\L^2(\F)$,
we set $\jump{v}_\F=\average{v}_\F=v$;
for an internal facet~$\F$
with fixed normal vector $\nuboldF$,
shared by~$\Tplus$ and~$\Tminus$
with outward unit normal vectors~$\nubold_{\Tplus}$
and~$\nubold_{\Tminus}$,
given $v$ in $\H^1(\tauh,\Omega)$ with restrictions
$v_{\Tplus}$ and~$v_{\Tminus}$ over~$\Tplus$ and~$\Tminus$, we set
\[
\jump{v}_{\F} :=
(v_{\Tplus} \nuboldF \cdot \nubold_{\Tplus}
    + v_{\Tminus} \nuboldF \cdot \nubold_{\Tminus})_{|\F},
\qquad\qquad\qquad
\average{v}_{\F} = \frac12 (v_{\Tplus} + v_{\Tminus})_{|\F}.
\]
We employ the same notation for these operators while applied to vector fields
and remove the subscript~$\F$ when it is clear from the context.

We introduce the following skeletal inner product and norm:
for any subset~$\widetilde{\Fcal}_{\textrm h}$ of~$\Fcaln$,
\[
(u,v)_{\widetilde{\Fcal}_{\textrm h}} 
= \sum_{\F \in \widetilde{\Fcal}_{\textrm h}} (u,v)_{\F} ,
\qquad\quad \Norm{v}{\widetilde{\Fcal}_{\textrm h}}^2 := (v, v)_{\widetilde{\Fcal}_{\textrm h}}
\qquad\qquad\quad \forall v \in \L^2(\F),
\; \forall \F \in \widetilde{\Fcal}_{\textrm h}.
\]
All differential operators defined piecewise over the mesh elements
are denoted with the same symbol of the original operator and a subscript~$\h$;
for instance, the piecewise Hessian over~$\tauh$ is denoted by~$\Dtwoh$.

\paragraph*{Some finite element spaces.}
In what follows, $\VbfHoh$
and~$\Vh$ denote
the standard vector Lagrange finite element space of order~$\p$
and  the standard DG space of order~$\p$.

\paragraph{The method, a lifting operator, and a generalized Hessian}.
Let~$\csigma$ and~$\ctau$ be two stabilization parameters
to be chosen sufficiently large;
see Lemma~\ref{lemma:coercivity-continuity} below.
We introduce the IPDG bilinear form
\small\begin{equation} \label{IPDG-bilinear-form-noL}
\begin{split}
\Bh(\uh,\vh)
:= & (\Dtwoh \uh, \Dtwoh \vh)_{\Omega}
    + \Big(\jump{\uh}, \average{\nuboldF^T \divbf\Dtwoh \uh} \Big)_{\Fcaln}
    - \Big(\jump{\nabla \uh}, \average{\nuboldF^T \Dtwoh \uh}\Big)_{\Fcaln} \\
   & + \Big(\average{\nuboldF^T \divbf\Dtwoh\uh}, \jump{\vh} \Big)_{\Fcaln}
     - \Big(\average{\nuboldF^T \Dtwoh\uh}, \jump{\nabla \vh}\Big)_{\Fcaln} \\
   & + \csigma (\h^{-3}\jump{\uh},\jump{\vh})_{\Fcaln}
     + \ctau (\h^{-1}\jump{\nablah \uh},\jump{\nablah \vh})_{\Fcaln}
\qquad\qquad \forall \uh,\vh\in\Vh.
\end{split}
\end{equation}\normalsize
We introduce the lifting operator~$\Lbbh : \Vh \oplus V \to [\Vh]^{2\times 2}$ as
\small\begin{equation} \label{lifting}
(\Lbbh (v), \Abbh)_{\Omega}
:=  \Big(\jump{v}, \average{\nuboldF^T \divbf\Abbh} \Big)_{\Fcaln}
   - \Big(\jump{\nabla v}, \average{\nuboldF^T \Abbh}\Big)_{\Fcaln} 
   \qquad \forall \Abbh \in [\Vh]^{2\times2},
\end{equation} \normalsize
which satisfies the following result;
see \cite[Lemma~5.1]{Georgoulis-Houston:2009}.
\begin{lemma} \label{lemma:properties-Lcal}
The lifting operator in~\eqref{lifting}
is well-defined.
Moreover, there exists a positive constant~$\cLcalh$
only depending on~$\csigma$ and~$\ctau$
in~\eqref{IPDG-bilinear-form-noL},
the shape-regularity parameter~$\gamma$, and~$\p$ such that
\[
\Norm{\Lbbh(\vh)}{\Omega}
\le \cLcalh \left( \h^{-\frac32}\Norm{\jump{\vh}}{\Fcaln}
                    + \h^{-\frac12}\Norm{\jump{\nabla\vh}}{\Fcaln} \right)
\qquad\qquad \forall \vh \in \Vh.
\]
If~$v$ belongs to $V$, then $\Lbbh (v) = [0]^{d\times d}$.
\end{lemma}
Using~\eqref{lifting}, the bilinear form
in~\eqref{IPDG-bilinear-form-noL}
can be equivalently rewritten as
\begin{equation} \label{IPDG-bilinear-form}
\begin{split}
\Bh(\uh,\vh)
:= & (\Dtwoh \uh, \Dtwoh \vh)_{\Omega}
    + (\Lbbh(\uh),\Dtwoh \vh)_{\Omega} 
   + (\Dtwoh \uh, \Lbbh(\vh))_{\Omega}\\
   & + \csigma (\jump{\uh},\jump{\vh})_{\Fcaln}
     + \ctau (\jump{\nablah \uh},\jump{\nablah \vh})_{\Fcaln}
\qquad\qquad \forall \uh,\vh\in\Vh.
\end{split}
\end{equation}
The bilinear form in~\eqref{IPDG-bilinear-form}
is well-defined also on $[\H^2_0(\Omega)]^2$
whereas that in~\eqref{IPDG-bilinear-form-noL} is not.

Introduce the DG norm
\begin{equation} \label{DG-norm}
\Norm{\vh}{\DG}^2
:= \Norm{\Dtwoh \vh}{\Omega}^2 
   + \Norm{\h^{-\frac32}\jump{\vh}}{\Fcaln}^2
   + \Norm{\h^{-\frac12}\jump{\nabla \vh}}{\Fcaln}^2.
\end{equation}
The following result is valid; see \cite[Lemma~5.2]{Georgoulis-Houston:2009}.
\begin{lemma} \label{lemma:coercivity-continuity}
Let the stabilization parameters~$\sigma$ and~$\ctau$ in~\eqref{IPDG-bilinear-form}
be sufficiently large.
The following coercivity and continuity properties hold true:
there exist positive constant~$c_{\flat}$ and~$c_{\sharp}$
only depending on~$\csigma$ and~$\ctau$
in~\eqref{IPDG-bilinear-form-noL},
the shape-regularity parameter~$\gamma$, and~$\p$ such that
\[
\Bh(\vh,\vh) \ge c_{\flat} \Norm{\vh}{\DG}^2,
\qquad\qquad\qquad
\Bh(\uh,\vh) \le c_{\sharp} \Norm{\uh}{\DG} \Norm{\vh}{\DG} .
\]
\end{lemma}
We are now in a position to introduce an IPDG method
for the discretization of problem~\eqref{problem:weak}:
\begin{equation} \label{method}
\begin{cases}
    \text{find } \uh \in \Vh \text{ such that}\\
    \Bh(\uh,\vh) = (\f,\vh)_{\Omega} \qquad\qquad\forall\vh\in\Vh.
\end{cases}
\end{equation}
The well-posedness of method~\eqref{method}
follows from Lax-Milgram lemma, Lemma~\ref{lemma:coercivity-continuity},
Cauchy-Schwarz' inequality on the right-hand side of~\eqref{method},
and the Poincar\'e inequality
for piecewise $\H^2$ functions in~\cite{Brenner-Wang-Zhao:2004}.

Given~$\uh$ the solution to~\eqref{method}
and $\Lbbh$ as in~\eqref{lifting},
we introduce the generalized Hessian
\begin{equation} \label{generalized:Hessian}
\Hbbhuh := \Dtwoh \uh + \Lbbh(\uh)
\end{equation}
and rewrite,
for all~$\uh$ and~$\vh$ in~$\Vh$,
the DG bilinear form in~\eqref{IPDG-bilinear-form} as
\small\[
\Bh(\uh,\vh)
= (\Hbb(\uh),\Hbb(\vh))_{\Omega}
    - (\Lbb(\uh),\Lbb(\vh))_{\Omega}
    + \csigma (\jump{\uh},\jump{\vh})_{\Fcaln}
    + \ctau (\jump{\nablah \uh},\jump{\nablah \vh})_{\Fcaln}.
\]\normalsize
\paragraph*{A class of generalized Hessians.}
The forthcoming analysis applies
to all piecewise polynomial tensor fields~$\Hbbh$
satisfying the following properties:
\begin{subequations} \label{properties-GH}
\begin{align}
(\Hbbh,\Dtwo \wh)_{\Omega} = (\f,\wh) _{\Omega}
& \qquad\qquad\qquad\qquad \forall 
    \wh \in \Vh\cap  \H^2_0(\Omega),
        \label{properties-GH:1}\\
(\Hbbh, \Abbh)_{\Omega}=0
& \qquad\qquad\qquad\qquad \forall \Abbh \in \underline{\mathbf P}_{\p}(\taun) \cap \Hbb(\divdivbfz,\Omega). \label{properties-GH:2}
\end{align}    
\end{subequations}
We refer to tensor fields fulfilling~\eqref{properties-GH}
as generalized Hessians.
In Section~\ref{subsection:properties:GH} below,
we shall prove that~$\Hbbhuh$ in~\eqref{generalized:Hessian}
satisfies~\eqref{properties-GH}.

\begin{remark} \label{remark:C1-tests}
In~\eqref{properties-GH:1}, the space~$\Vh\cap  \H^2_0(\Omega)$
is the finite element space~$\VHth$
as in \cite{Morgan-Scott:1975}
for $p$ larger than or equal to~$5$.
Assuming that a given mesh~$\tauh$
is an HCT refinement of a coarser mesh~$\tautildeh$
(i.e., an element~$\widetilde T$ in $\tautildeh$
is refined into three elements of $\tauh$
obtained connecting the vertices of~$\widetilde T$
to its centroid),
the space~$\Vh\cap  \H^2_0(\Omega)$
contains the HCT finite element
space~\cite{Douglas-Dupont-Percell-Scott:1979}
of degree~$\p=3$ and~$4$ on the mesh~$\tautildeh$.
\eremk
\end{remark}

\begin{remark} \label{remark:second-assumption-GH}
Observing that $\divdivbf (\symcurlbb)\cdot=0$,
property~\eqref{properties-GH:2} implies
\begin{equation} \label{property-GH-2-bis}
(\Hbbh, \Abbh)_{\Omega}=0
\qquad \text{where } \Abbh=\symcurlbb\thetabfh
\qquad\qquad \forall \thetabfh \in \VbfHoh,
\end{equation}
where we recall that~$\VbfHoh$ denotes the standard
vector $\H^1$-conforming finite element space of order~$\p+1$.
\eremk
\end{remark}

\section{A splitting of the generalized Hessian error} \label{section:setting}

Let~$\Hbbh$ be any generalized Hessian satisfying~\eqref{properties-GH}.

\subsection{A regular decomposition for the $\symcurlbb$ operator} \label{subsection:preliminari-error-split}
We present here a result,
which will be instrumental in deriving the
crucial identity~\eqref{rewriting-in-terms-of-max} below.
We postpone its proof to Appendix~\ref{appendix:nontrivial-topologies} below.
\begin{lemma} \label{lemma:special-split}
For all~$\Abb$ in $\Hbb_{\Sbb}(\divdivbfz,\Omega)$
there exist~$\Qbbh$ in $\RTbb_1(\taun) \cap \Hbb_{\Sbb}(\divdivbfz,\Omega)$
and~$\thetabf$ in~$\Hbf^1(\Omega)$ such that
\begin{equation} \label{eq:special-split-2}
\Abb = \symcurlbb\thetabf + \Qbbh.
\end{equation}
In particular, we have
\begin{equation} \label{RT-vanishes}
(\Hbbh,\Abb)
\overset{\eqref{eq:special-split-2}}{=}
    (\Hbbh,\symcurlbb\thetabf) + (\Hbbh,\Qbbh) 
\overset{\eqref{properties-GH:2}}{=}
    (\Hbbh,\symcurlbb\thetabf).
\end{equation}
\end{lemma}

\subsection{The error splitting}
\label{subsection:error-split}
Given~$s$ in~$V$ such that
\begin{equation} \label{definition:s}
s := \argmin_{t \in V} \Norm{\Hbbh - \Dtwo t}{\Omega}
\qquad\Longleftrightarrow\qquad
(\Dtwo s -\Hbbh, \Dtwo v)=0 \quad\forall v\in V,
\end{equation}
Pythagoras' theorem implies
\begin{equation} \label{Pythagoras}
\Norm{\Dtwo u - \Hbbh}{\Omega}^2
= \Norm{\Dtwo (u - s)}{\Omega}^2
  + \Norm{\Dtwo s - \Hbbh}{\Omega}^2
=: T_1^2 + T_2^2  .
\end{equation}
As for the term~$T_1$, a duality argument reveals that
\begin{equation} \label{T1-as-sup}
T_1 
= \sup_{v\in V,\ \SemiNorm{v}{2,\Omega}=1} (\Dtwo(u-s), \Dtwo v)_{\Omega}
\overset{\eqref{definition:s}}{=}
    \sup_{v\in V,\ \SemiNorm{v}{2,\Omega}=1} (\Dtwo u -\Hbbh, \Dtwo v)_{\Omega} .
\end{equation}
We also rewrite the term~$T_2$ as a sup
over a suitable space. To this aim, we need a technical result.

\begin{proposition} \label{proposition:2nd-term-divdiv-zero}
The tensor $\Dtwo s - \Hbbh(\uh)$
belongs to $\Hbb(\divdivbf^{\zerobf},\Omega)$
in~\eqref{Hdivdiv-spaces}.
\end{proposition}
\begin{proof}
By definition~\eqref{definition:s}, we have
\[
0
= (\Bbb, \Dtwo v)_{\Omega}
= _{-2}\langle \divdivbf \Bbb, v \rangle_2
\qquad\qquad\qquad
\forall v\in V.
\]
\end{proof}

Proposition~\ref{proposition:2nd-term-divdiv-zero} implies
that the space~$\Hbb(\divdivbf^{\zerobf},\Omega)$
is Hilbert with respect to the $\L^2(\Omega)$ norm.
Hence, a duality argument entails the following rewriting
of the term~$\T_2$ in~\eqref{Pythagoras}:
\begin{equation} \label{T2-as-sup}
\begin{split}
T_2 
&   = \sup_{\wbb\in \H(\divdivbf^0,\Omega),\ \Norm{\wbb}{\Omega}=1}
       (\Hbbh - \Dtwo s, \wbb)_{\Omega} \\
&   \overset{\eqref{eq:special-split-2}, \eqref{RT-vanishes}}{=}
        \sup_{\psibf \in \Hbf^1(\Omega),\ 
        \Norm{\symcurlbb \psibf}{\Omega}=1}
       (\Hbbh - \Dtwo s, \symcurlbb \psibf)_{\Omega} \\
&   \overset{\Dtwo s=\sym (\Dtwo s), (\text{IBP})}{=}
    \sup_{\psibf \in \Hbf^1(\Omega),\ 
       \Norm{\symcurlbb \psibf}{\Omega}=1}
       (\Hbbh, \symcurlbb \psibf)_{\Omega} .
\end{split}
\end{equation}
Combining~\eqref{Pythagoras},
\eqref{T1-as-sup}, and~\eqref{T2-as-sup},
we obtain the following crucial identity
on domains~$\Omega$ with nontrivial topology:
\small{\begin{equation} \label{rewriting-in-terms-of-max}
\begin{aligned}
\Norm{\Dtwo u-\Hbbh}{\Omega} ^2
& = \Big(\sup_{v\in V,\ \SemiNorm{v}{2,\Omega}=1}  (\Dtwo u -\Hbbh, \Dtwo v)_{\Omega} \Big)^2 \\
& \quad +  \Big(\sup_{\psibf \in \Hbf^1(\Omega),\ 
            \Norm{\symcurlbb \psibf}{\Omega}=1}
       (\Hbbh, \symcurlbb \psibf)_{\Omega}) \Big)^2.
\end{aligned}
\end{equation}}\normalsize

\begin{remark} \label{remark:kernel-EM}
The error measure on the left-hand side
of~\eqref{rewriting-in-terms-of-max}
is an $\L^2$ measure of the approximation
of $\Hbbh(\uh)$ to the Hessian of~$u$.
Similar to the Poisson case~\cite{ChaumontFrelet:2025},
this does not guarantee any control
on the approximation of~$\uh$ to~$u$ itself,
unless the jumps are lifted to
sufficiently high-order polynomial
spaces, cf. \cite{John-Neilan-Smears:2016}.
\eremk
\end{remark}

\section{A posteriori bounds for a class of generalized Hessians} \label{section:apos-bounds}

Aim of this section is to derive upper and lower a posteriori error bounds
for the error in~\eqref{rewriting-in-terms-of-max}
in terms of the error estimator
\begin{subequations}
\label{error-estimator}
\begin{equation} \label{global-error-estimator}
    \eta^2 := \sum_{\T\in\tauh} \etaT^2,
\end{equation}
where,
for any generalized Hessian~$\Hbbh$ satisfying~\eqref{properties-GH},
\begin{align}
\nonumber
\etaT^2
&   := \sum_{j=1}^5 \eta_{j,\T}^2
    :=  \hT^4 \Norm{\f-\div\divbf \Hbbh}{\T}^2\\
\nonumber
& \quad + \frac12 \sum_{\F\in\FcalT\cap\FcalnI}
    \Big( \hF\Norm{\nuboldF^T\jump{\Hbbh}\nuboldF}{\F}^2
    + \hF^3 \Norm{\partialtbfF(\tauboldF^T \jump{\Hbbh}\nuboldF)
    + \nuboldF^T \jump{\divbf\Hbbh}}{\F}^2 \Big) \\
& \quad + \hT^2 \Norm{\curlsym\Hbbh}{\T}^2
     + \sum_{\F\in\FcalT} \hF \Norm{\tauboldF^T  \jump{\sym\Hbbh}}{\F}^2.
\label{local-error-estimator}
\end{align}
\end{subequations}
The local error estimators in~\eqref{local-error-estimator}
do not depend explicitly on the stabilization of the DG method;
if $\Hbbh$ is that in~\eqref{generalized:Hessian},
then the stabilization term is encapsulated in~$\Hbbhuh$.
The first three terms and the last two terms
on the right-hand side of~\eqref{local-error-estimator}
are residual-type terms related to the
conforming ($T_1$ in~\eqref{rewriting-in-terms-of-max})
and nonconforming ($T_2$ in~\eqref{rewriting-in-terms-of-max})
parts of the error, respectively.

\begin{remark} \label{remark:powers-h}
All the powers of~$\h$ in~\eqref{local-error-estimator} are positive.
\eremk
\end{remark}

The main result of this section is reported here;
its proof is a given in
Sections~\ref{subsection:upper-lower-bound-T1}
and~\ref{subsection:upper-lower-bound-T2}.

\begin{theorem} \label{theorem:apos-bounds}
Given~$u$ and~$\uh$ the solutions
to~\eqref{problem:weak} and~\eqref{method},
let~$\Hbbh(\uh)$ be the generalized Hessian as in~\eqref{generalized:Hessian},
and~$\eta$ and $\etaT$ be the global and local error estimators as
in~\eqref{global-error-estimator} and~\eqref{local-error-estimator}.
Furthermore, let~$\fh$ denote the piecewise $\L^2$ projection
onto the space of polynomials of degree~$\p$.
Then, the following reliability and efficiency error bounds are valid:
there exist positive constants $C_{\rm rel}$ and $C_{\rm eff}$
independent of~$\h$,
but depending on~$\gamma$ in~\eqref{def:gamma},
the shape of the vertex patches in the mesh,
and the polynomial degree~$\p$, such that
\begin{subequations} \label{apos-bound}
\begin{align}
& \Norm{\Dtwo u - \Hbbh(\uh)}{\Omega}
\le C_{\rm rel} \eta, \label{apos-bound-1}\\
& \etaT \le C_{\rm eff} \big(\Norm{\Dtwoh u - \Hbbh(\uh)}{\omegaT}
            + \sum_{\Ttilde \in \omegaT}
             \h_{\Ttilde}^2 \Norm{\f-\fh}{\Ttilde} \big)
             \qquad\qquad \forall \T \in \tauh \label{apos-bound-2}.
\end{align}
\end{subequations}
Bound~\eqref{apos-bound-2} holds true for any polynomial degree~$\p$
larger than or equal to~2.
Bound~\eqref{apos-bound-1} is valid assuming
$\p$ larger than or equal to~$5$ on any mesh,
and $\p=3$ and~$4$ on meshes stemming from HCT refinements
as discussed in Remark~\ref{remark:C1-tests}.
\end{theorem}
The numerical results in Section~\ref{section:nr} below
indicate that bound~\eqref{apos-bound-1}
remains valid for $\p=2,3,4$ on any mesh.

\begin{remark} \label{remark:p-dependence-bounds}
It is possible to derive the explicit dependence of~$\crel$ and~$\ceff$
in~\eqref{apos-bound-1} and~\eqref{apos-bound-2}
in terms of~$\p$ following~\cite{Dong-Mascotto-Sutton:2021};
the corresponding proof boils down
to track the polynomial dependence in the polynomial inverse estimates
as in Section~\ref{subsection:technical-tools} below.
As discussed in~\cite{ChaumontFrelet:2025},
we expect the $\p$-dependence to be milder than
the corresponding one in the standard residual a posteriori
error analysis for IPDG schemes.
\eremk
\end{remark}

The remainder of this section is organized as follows.
After introducing technical tools
in Section~\ref{subsection:technical-tools}
we derive a posteriori error bounds
in Sections~\ref{subsection:upper-lower-bound-T1}
and~\ref{subsection:upper-lower-bound-T2},
for any generalized Hessian satisfying
the properties detailed in~\eqref{properties-GH}.
We prove that the generalized Hessian in~\eqref{generalized:Hessian}
satisfies~\eqref{properties-GH} in Section~\ref{subsection:properties:GH}.
In order to shorten the notation, we shall write $\Hbbh$.
A comparison with the standard IPDG residual error estimators
is given in Section~\ref{subsection:comparison-standard-ee}.

\subsection{Technical tools} 
\label{subsection:technical-tools}
This section is devoted to discussing several technical
tools, which will be instrumental
in the a posteriori error analysis.
Notably, we discuss:
the approximation properties
of $\mathcal C^0$ and $\mathcal C^1$
quasi-interpolants,
and a Korn-type inequality for the $\symcurlbb$ operator,
which are essential in deriving the reliability of the error estimator;
extension operators, bubble functions,
and polynomial inverse estimates,
which are needed in deriving the efficiency of the error estimator.

\paragraph*{$\mathcal C^0$ quasi-interpolants.}
Introduce the vector Cl\'ement
quasi-interpolant~$\thetabfI$
of~$\thetabf$ in $\Hbf^1(\Omega)$
as in~\cite{Ciarlet:1975}.
We have the standard approximation properties
\begin{equation} \label{Clement-qi}
\Norm{\thetabf-\thetabfI}{\T}
+ \hT \SemiNorm{\thetabf-\thetabfI}{1,\T}
\lesssim \hT \SemiNorm{\thetabf}{1,\omegaT},
\qquad\qquad
\Norm{\thetabf-\thetabfI}{\F}
\lesssim \hT^\frac12 \SemiNorm{\thetabf}{1,\omegaF}.
\end{equation}

\paragraph*{$\mathcal C^1$ quasi-interpolants
preserving nodal values.}
For $\p$ larger than or equal to~$5$,
consider the quasi-interpolant~$\vI$ in~$\VHth$
of a function~$v$~in $V$
as in \cite[Theorem~3.2]{Graessle:2022},
which is a modification of that in
\cite[Section~5]{Girault-Scott:2002}
and allows for exact pointwise interpolation
at the vertices of the mesh;
see also \cite[Theorem~2]{Carstensen-Hu:2021}.
The following interpolation estimates
can be derived from the results in \cite[Theorem~3.2]{Graessle:2022}
and standard trace inequalities:
for all elements~$\T$ and all~$\F$ in~$\FcalT$,
\begin{equation} \label{interpolation:C1}
\hT^{-2}  \Norm{v-\vI}{\T}
+ \sum_{\F \in \FcalT}
   \Big( \hF^{-\frac32} \Norm{v-\vI}{\F}
        + \hF^{-\frac12} \Norm{\nabla(v-\vI)}{\F}
   \Big)
\lesssim  \SemiNorm{v}{2,\omegaT}.
\end{equation}
Compared to other $\mathcal C^1$ quasi-interpolants,
e.g., that in \cite{Girault-Scott:2002},
the one we pick has the property
of interpolating functions at the vertices of the mesh
\cite[Theorem 3.2 (a)]{Graessle:2022}:
\begin{equation} \label{C1-vertices}
    (v-\vI)(\nu)=0
    \qquad\qquad\qquad
    \forall \nu \text{ vertex of $\tauh$}.
\end{equation}
With a similar reasoning,
for $\p=3,4$, and meshes as those discussed in
Remark~\ref{remark:C1-tests},
one can define an HCT quasi-interpolant
satisfying~\eqref{interpolation:C1}
and~\eqref{C1-vertices}.
In fact, the interpolant in~\cite{Girault-Scott:2002}
is based on nodal averaging employing
auxiliary $\L^2$ dual basis on patches/elements.
If $v$ belongs to $\H^2(\Omega)$, we may replace
averages with point values at the pointwise vertex
degrees of freedom.
This newly defined operator preserves polynomials
and is continuous in $\H^2(\Omega)$,
which is enough to derive the interpolation estimates.

\paragraph*{A Korn-type inequality for the $\symcurlbb$ operator.}
We state here a technical result,
which will be instrumental in deriving the reliability estimates
in Section~\ref{subsection:upper-lower-bound-T2} below.
It is adapted from \cite{Dorsek-Melenk:2013}
and its proof is postponed
to Appendix~\ref{appendix:tools-nonconforming} below.
We refer to it as a Korn-type inequality,
since $\symcurlbb \psibf$ has the same $\L^2$ norm
as that of the symmetric gradient of~$\phibf = (\psi_2,-\psi_1)$.

\begin{proposition} \label{proposition:tools-nc}
There exists a positive constant~$\CNL$ only depending
on the shape of the vertex patches of~$\taun$
and the shape-regularity parameter of the mesh
such that, for all $\psibf$ in $\Hbf(\symcurlbb,\Omega)$
(which coincides with $\Hbf^1(\Omega)$),
there exists $\psibfh$ in $\VHoh$ with
\begin{equation} \label{eq:tools-nc}
\SemiNorm{\psibf-\psibfh}{1,\Omega}
\le \CNL
    \Norm{\symcurlbb\psibf}{\Omega}.
\end{equation}
\end{proposition}

\begin{remark}
As discussed in \cite[Remark~5]{Dorsek-Melenk:2013}
in an analogous context, the dependence of the constant~$\CNL$
in~\eqref{eq:tools-nc} is difficult to check in practice
on general meshes.
However, if newest-vertex bisection is used,
then only a finite number of patch-types are produced
by the refinement procedure.
\eremk
\end{remark}

\paragraph*{A facet extension operator.}
We shall be considering
extension operators $\Ecal_\F = \Ecal : \L^2(\F) \to \L^2(\omegatildeF)$
for all~$\F$ in~$\Fcaln$.
Such operators are defined so as to satisfy
\begin{equation} \label{extension-operator}
\Ecal v (x) = v(x') \text{ in } \T, \quad \T \subset \omegaF,
\text{ if } x'\in \F \text{ is such that }
\lambda^\T(x') = \lambda^\T(x),
\end{equation}
where $\lambda^\T$ is any of the two barycentric
coordinates of~$\T$, which is not associated
with the facet~$\F$, cf. \cite{Verfurth:2013}.
In particular, this operator acts as the identity on the facet~$\F$.

\paragraph*{$\Ccal^0$ bubble functions and inverse estimates.}
In the treatment of the lower bound for the nonconforming
part of the error, see Section~\ref{subsection:upper-lower-bound-T2},
we shall be using two different bubbles.
The first one is defined on the element~$\T$ and is given by
\begin{equation} \label{C0-bubble-bulk}
    \btildeT:=
    27 \ \Pi _{j=1}^3 \lambda_j^\T,
\end{equation}
where the~$\lambda_j^\T$, $j=1,2,3$,
are the barycentric coordinates of~$\T$.
We have \cite{Verfurth:2013}
the standard inverse estimates
\begin{equation} \label{C0-bubble-bulk-IE}
\Norm{\qp}{\T} 
    \lesssim \Norm{\btildeT^\frac12 \qp}{\T},
\qquad
\SemiNorm{\btildeT \qp}{1,\T} 
    \lesssim \hT^{-1} \Norm{\qp}{\T} 
\qquad\qquad\qquad
\forall \qp \in P_\p(\T).
\end{equation}
The second one is defined on the patch~$\omegatildeF$,
see~\eqref{patches},
for a given facet~$\F$ and is given by
\begin{equation} \label{C0-bubble-facet}
\btildeF{}_{|\T} = 4\ \Pi_{j=1}^2 \lambda_j^\T
\qquad\qquad\qquad \forall \T \subset \omegatildeF,
\end{equation}
where $\lambda_j^\T$, $j=1,2,3$,
are the barycentric coordinates of~$\T$
subset of~$\omegatildeF$,
$j=3$ being the index related to the facet~$\F$.
Given $\Ecal$ as in~\eqref{extension-operator},
we have \cite{Verfurth:2013} the inverse estimates
\small\begin{equation} \label{C0-bubble-facet-IE}
\Norm{\qp}{\F} 
    \lesssim \Norm{\btildeF^\frac12 \qp}{\F},
\quad
\SemiNorm{\btildeF \Ecal(\qp)}{1,\T} 
+ \hT^{-1}\Norm{\btildeF \Ecal(\qp)}{\T} 
    \lesssim \hF^{\frac12} \Norm{\qp}{\F} 
\qquad  \forall \qp \in P_\p(\F).
\end{equation} \normalsize

\paragraph*{$\Ccal^1$ bubble functions and inverse estimates.}
In the treatment of the lower bound for the conforming
part of the error, see Section~\ref{subsection:upper-lower-bound-T1},
we shall be using three different bubbles.
The first one is defined on the element~$\T$ and,
for $\btildeT$ as in~\eqref{C0-bubble-bulk},
is given by
\begin{equation} \label{C1-bubble-bulk}
    \bT:= \btildeT^2.
\end{equation}
In particular, this is a $\Ccal^0$ bubble
with vanishing normal component
of the gradient and trace over~$\partial\T$.
We have \cite{Verfurth:2013}
the standard inverse estimates
\begin{equation} \label{C1-bubble-bulk-IE}
\Norm{\qp}{\T} 
    \lesssim \Norm{\bT^\frac12 \qp}{\T},
\qquad
\SemiNorm{\bT \qp}{2,\T} 
    \lesssim \hT^{-2} \Norm{\qp}{\T} 
\qquad\qquad\qquad
\forall \qp \in P_\p(\T).
\end{equation}
The second one is associated with the facet~$\F$ and,
for $\btildeF$ as in~\eqref{C0-bubble-facet}, is given by
\begin{equation} \label{C1-bubble-facet}
    \bF:= \btildeF^2.
\end{equation}
In particular, this is a $\Ccal^1$ bubble
vanishing and with vanishing normal component
of the gradient over~$\partial\omegatildeF$.
We have \cite{Verfurth:2013}
the standard inverse estimates:
given $\Ecal$ as in~\eqref{extension-operator},
\begin{equation} \label{C1-bubble-facet-IE}
\begin{split}
& \Norm{\qp}{\F} 
    \lesssim \Norm{\bF^\frac12 \qp}{\F},
\qquad
\hF^2 \SemiNorm{\bF \Ecal(\qp)}{2,\T} 
+ \Norm{\bF \Ecal(\qp)}{\T} 
    \lesssim \hF^{\frac12} \Norm{\qp}{\F} ,\\
& \Norm{\partial_{\nuboldF} (\bF \Ecal(\qp))}{\F}
    \lesssim \hF^{-1} \Norm{\qp}{\F}
\qquad\qquad  \forall \qp \in P_\p(\F).
\end{split}
\end{equation}
The third one is less standard,
see \cite[Lemma 4.3]{Brenner-Gudi-Sung:2010}
and \cite[Lemma 6]{Carstensen-Gallistl-Gedicke:2019},
and is defined on the patch~$\omegatildeF$ as in~\eqref{patches}.
Given $\Ecal$ as in \eqref{extension-operator} and
$\bstarF = \btildeF^2 \nuboldF \cdot (\xbf-\xbf_\F)$
where~$\xbf_\F$ denotes the midpoint of~$\F$,
cf. \cite[eq. (10)]{Carstensen-Gallistl-Gedicke:2019},
there holds
\begin{equation}\label{C1-bubble-facet*}
    \bstarF{}_{|\partial\omegaF\cup\F}=0,
    \qquad\qquad
    \Norm{\qp}{\F} \lesssim \Norm{( \partial_{\nuboldF}\bstarF)^\frac12 \qp}{\F}
    \quad \forall \qp \in P_\p(\F).
\end{equation}
For all $\qp$ in $P_\p(\F)$,
there exists a function~$\wF$
in $\H^2_0(\omegatildeF) \subset \H^2_0(\Omega)$
piecewise in $P_{\p+4}(\T)$,
$\T$ in $\omegatildeF$, such that
\small\begin{equation}\label{C1-bubble-facet*-IE-1}
\begin{split}
&   \wF:= \bstarF \Ecal(\qp)
    \text{ on } \omegaF,
    \qquad\qquad
    \partial_{\nuboldF} \wF
    = \partial_{\nuboldF}\bstarF \Ecal(\qp)
        + \bstarF \partial_{\nuboldF} \Ecal(\qp) 
    = \partial_{\nuboldF}\bstarF \Ecal(\qp)
    \text{ on } \F, \\
&   \wF = 0 \text{ on } \F,
    \qquad\qquad
    \Norm{\wF}{\omegatildeF} + \hF^2\SemiNorm{\wF}{2,\omegatildeF}
    \lesssim \hF^\frac32 \Norm{\qp}{\F} .
\end{split}
\end{equation}

\subsection{Upper and lower bounds of the conforming part of the error} 
\label{subsection:upper-lower-bound-T1}

\paragraph*{The upper bound.}
For all~$v$ in~$V$ with $\SemiNorm{v}{2,\Omega}=1$
and given the $\Ccal$ quasi-interpolant~$\vI$ as in~\eqref{interpolation:C1},
we have
\[
\begin{split}
& (\Dtwo(u-s), \Dtwo v)_{\Omega}
 \overset{\eqref{definition:s}}{=}
 (\Dtwo u - \Hbbh, \Dtwo v)_{\Omega}
  \overset{\eqref{properties-GH:1}}{=}
    (\Dtwo u - \Hbbh, \Dtwo (v-\vI))_{\Omega} \\
& \overset{\eqref{Green-divdiv}}{=}
    \sum_{\T\in\tauh} (\f-\div\divbf \Hbbh, v-\vI)_{\T}
    - \sum_{\F\in\FcalnI} (\nuboldF^T \jump{\Hbbh}, \nabla(v-\vI))_{\F}
    + \sum_{\F\in\FcalnI} (\nuboldF^T \jump{\divbf\Hbbh}, v-\vI)_{\F}\\
& \overset{\eqref{C1-vertices}}{=}
    \sum_{\T\in\tauh} (\f-\div\divbf \Hbbh, v-\vI)_{\T}
    - \sum_{\F\in\FcalnI} (\nuboldF^T\jump{\Hbbh}\nuboldF,
    \nuboldF^T\nabla(v-\vI))_{\F} \\
& \qquad
    + \sum_{\F\in\FcalnI} (\partialtbfF(\tauboldF^T \jump{\Hbbh}\nuboldF)
    + \nuboldF^T \jump{\divbf\Hbbh}, v-\vI)_{\F}.
\end{split}
\]\normalsize
Further using Cauchy-Schwarz' inequality
(first for integrals, then for sequences)
and the approximation properties~\eqref{interpolation:C1}
of $\mathcal C^1$ quasi-interpolants,
and recalling that $\SemiNorm{v}{2,\Omega}=1$
yield
\[
\begin{split}
(\Dtwo(u-s), \Dtwo v)_{\Omega}
& \lesssim  \Big[  \sum_{\T\in\tauh} \hT^4 \Norm{\f-\div\divbf \Hbbh}{\T}^2
    + \sum_{\F\in\FcalnI} \Big( \hF \Norm{\nuboldF^T\jump{\Hbbh}\nuboldF}{\F}^2 \\
& \qquad     
    + \hF^3 \Norm{\partialtbfF(\tauboldF^T \jump{\Hbbh}\nuboldF)
    + \nuboldF^T \jump{\divbf\Hbbh}}{\F}^2 \Big)
    \Big]^\frac12 
    \overset{\eqref{local-error-estimator}}{\le}
    \eta.
\end{split}
\]

\paragraph*{The lower bound.}
We begin with estimating the elemental terms.
Given $\bT$ as in~\eqref{C1-bubble-bulk}, we observe that
\[
\begin{split}
\Norm{\fh-\div\divbf \Hbbh}{\T}^2
& \lesssim (\fh-\div\divbf \Hbbh, \underbrace{\bT^2 (\fh-\div\divbf \Hbbh)}_{=:\wT})_{\T} \\
& \overset{\eqref{problem:weak}}{=} 
    (\f-\fh,\wT)_{\T} + (\Dtwo u - \Hbbh, \Dtwo \wT)_{\T} \\
& \le \Norm{\f-\fh}{\T} \Norm{\wT}{\T} +  \Norm{\Dtwo u - \Hbbh}{\T} \Norm{\Dtwo\wT}{\T}\\
& \overset{\eqref{C1-bubble-bulk-IE}}{\lesssim} 
    \Norm{\f-\fh}{\T} \Norm{\wT}{\T} 
        + \Norm{\Dtwo u - \Hbbh}{\T} \hT^{-2} \Norm{\wT}{\T}\\
&   \overset{\eqref{C0-bubble-bulk}}{\le}
    (\Norm{\f-\fh}{\T} + \hT^{-2} \Norm{\Dtwo u - \Hbbh}{\T})
        \Norm{\fh-\div\divbf \Hbbh}{\T}. 
\end{split}
\]
For all~$\T$ in~$\tauh$, we deduce
\[
\hT^2 \Norm{\fh-\div\divbf \Hbbh}{\T}
\lesssim \hT^2 \Norm{\f-\fh}{\T} + \Norm{\Dtwo u - \Hbbh}{\T},
\]
and, by the triangle inequality,
\begin{equation} \label{lower-bound:c-bulk}
\hT^2 \Norm{\f-\div\divbf \Hbbh}{\T}
\lesssim \hT^2 \Norm{\f-\fh}{\T} + \Norm{\Dtwo u - \Hbbh}{\T} .
\end{equation}
Next, we focus on the normal-normal component term
appearing in~\eqref{local-error-estimator}.
Given $\Ecal$ and $\bstarF$ as in
\eqref{extension-operator} and~\eqref{C1-bubble-facet*}
(here $\nubold_{\F}^T \jump{\Hbbh} \nubold_{\F}$
plays the role of $\qp$), we write
\small\[
\begin{split}
& \Norm{\nubold_{\F}^T \jump{\Hbbh} \nubold_{\F}}{\F}^2
 \overset{\eqref{C1-bubble-facet*}}{\lesssim}
 (\nubold_{\F}^T \jump{\Hbbh} \nubold_{\F},
 \partial_{\nuboldF}\bstarF
    \Ecal(\nubold_{\F}^T \jump{\Hbbh} \nubold_{\F}))_{\F}
    \overset{\eqref{C1-bubble-facet*-IE-1}}{=}  
    (\nubold_{\F}^T \jump{\Hbbh} \nubold_{\F},
    \partial_{\nuboldF} \wF)_{\F}\\
& \overset{\eqref{C1-bubble-facet*}}{=}
    \sum_{j=1}^2
    (\nubold_{\Tj}^T \Hbbh \nubold_{\Tj},
    \partial_{\nubold_{\Tj}} \wF)_{\partial\Tj} 
 \overset{\eqref{C1-bubble-facet*-IE-1}}{=}
        -(\divdivbf \Hbbh, \wF)_{\omegatildeF}
        + (\Hbbh, \Dtwo \wF)_{\omegatildeF} \\
& \overset{\eqref{problem:weak}}{=}
    -(\divdivbf \Hbbh - f, \wF)_{\omegatildeF}
        + (\Hbbh-\Dtwo u, \Dtwo \wF)_{\omegatildeF} \\
&   \overset{\eqref{C1-bubble-facet*-IE-1}}{\lesssim}
    \big(\hF^2\Norm{\f-\divdivbf\Hbbh}{\omegatildeF}
    + \Norm{\Dtwo u - \Hbbh}{\omegatildeF} \big)
    \hF^{-\frac12} \Norm{\nubold_{\F}^T \jump{\Hbbh} \nubold_{\F}}{\F}.
\end{split}    
\] \normalsize
We simplify the display above into
\begin{equation} \label{bound:normal-normal}
\h^\frac12 \Norm{\nubold_{\F}^T \jump{\Hbbh} \nubold_{\F}}{\F}
\overset{\eqref{lower-bound:c-bulk}}{\lesssim}
    \hT^2 \Norm{\f-\fh}{\omegatildeF} 
        + \Norm{\Dtwo u - \Hbbh}{\omegatildeF}.
\end{equation}
Finally, we handle the facet terms involving the normal component of the divergence
and the normal-tangential component of the generalized Hessian.
Let~$\F$ be shared by the elements~$\To$ and~$\Ttw$.
Given $\Ecal$ and $\bF$ as in
\eqref{extension-operator} and~\eqref{C1-bubble-facet}, we have
\small\[
\begin{split}
& \Norm{\partialtbfF(\tauboldF^T \jump{\Hbbh}\nuboldF)
    + \nuboldF^T \jump{\divbf\Hbbh}}{\F}^2\\
&  \overset{\eqref{C1-bubble-facet-IE}}{\lesssim}
    (\partialtbfF(\tauboldF^T \jump{\Hbbh}\nuboldF)
    + \nuboldF^T \jump{\divbf\Hbbh},
    \underbrace{\bF \Ecal(\partialtbfF(\tauboldF^T \jump{\Hbbh}\nuboldF)
    + \nuboldF^T \jump{\divbf\Hbbh})}_{=:\wF})_{\F}\\
&  \overset{\eqref{Green-divdiv}}{=}
    \sum_{j=1}^2 \Big[ +(\divdivbf \Hbbh, \wF)_{\Tj}
        - ( \Hbbh, \Dtwo \wF)_{\Tj} \\
&  \qquad\qquad + \sum_{\Ftilde \in\Fcal^{\Tj}} \sum_{\nu \in \mathcal V^{\Ftilde}}
                (\taubold_{\Tj}^T \Hbbh \nubold_{\Tj})(\nu) 
                    \underbrace{\wF(\nu)}_{=0}
    + \sum_{\Ftilde \in \Fcal^{\Tj}}
    (\nubold_{\Tj}^T \Hbbh \nubold_{\Tj},
    \underbrace{\partial_{\nubold_{\Tj}} \wF}_{=0 \text{ if $\Ftilde \ne \F$}})_{\Ftilde} \Big]    \\
& \overset{\eqref{problem:weak}}{=}
    (-\f +\divdivbfh \Hbbh, \wF)_{\omegaF}
    + (-\Hbbh + \Dtwo u, \Dtwo \wF)_{\omegaF}
    + (\nubold_{\F}^T \jump{\Hbbh} \nubold_{\F},
    \partial_{\nubold_{\F}} \wF)_{\F}
 =: S_1 + S_2 +S_3.
\end{split}
\] \normalsize
We estimate the three terms on the right-hand side separately.
As for the term~$S_1$, we write
\[
\begin{split}
S_1
& \le \Norm{\f-\divdivbfh\Hbbh}{\omegaF} \Norm{\wF}{\omegaF}\\
& \overset{\eqref{C1-bubble-facet-IE}}{\lesssim}
    \Norm{\f-\divdivbfh\Hbbh}{\omegaF}
    \hF^{\frac12} \Norm{\partialtbfF(\tauboldF^T \jump{\Hbbh}\nuboldF)
    + \nuboldF^T \jump{\divbf\Hbbh}}{\F} \\
& \overset{\eqref{lower-bound:c-bulk}}{\lesssim}
    \hF^{-\frac32} \big[ \hT^2 \Norm{\f-\fh}{\omegaF}
                + \Norm{\Dtwo u - \Hbbh}{\omegaF} \big]
    \Norm{\partialtbfF(\tauboldF^T \jump{\Hbbh}\nuboldF)
    + \nuboldF^T \jump{\divbf\Hbbh}}{\F} .
\end{split}
\]
As for the term~$S_2$, we have
\small\[
\begin{split}
S_2
& \le \Norm{\Dtwo u - \Hbbh}{\omegaF} \Norm{\Dtwo \wF}{\omegaF} 
\overset{\eqref{C1-bubble-facet-IE}}{\lesssim}
    \hF^{-\frac32} \Norm{\Dtwo u - \Hbbh}{\omegaF}
          \Norm{\partialtbfF(\tauboldF^T \jump{\Hbbh}\nuboldF)
    + \nuboldF^T \jump{\divbf\Hbbh}}{\F} .
\end{split}
\]\normalsize
We now focus on the term~$S_3$:
\small\[
\begin{split}
S_3
& \le  \Norm{\nubold_{\F}^T \jump{\Hbbh} \nubold_{\F}}{\F}
        \Norm{\partial_{\nubold_{\F}} \wF}{\F}
\overset{\eqref{C1-bubble-facet-IE}}{\lesssim}
    \hF^{\frac12}\Norm{\nubold_{\F}^T \jump{\Hbbh} \nubold_{\F}}{\F}
            \hF^{-\frac32} \Norm{\partialtbfF(\tauboldF^T \jump{\Hbbh}\nuboldF)
    + \nuboldF^T \jump{\divbf\Hbbh}}{\F}.
\end{split}
\]\normalsize
Collecting the four displays above yields
\begin{equation} \label{lower-bound:c-facet-1}
\begin{split}
& \hF^{\frac32} \Norm{\partialtbfF(\tauboldF^T \jump{\Hbbh}\nuboldF)
    + \nuboldF^T \jump{\divbf\Hbbh}}{\F}\\
& \lesssim  \hT^2 \Norm{\f-\fh}{\omegaF} 
    + \Norm{\Dtwo u - \Hbbh}{\omegaF} 
    + \hF^{\frac12}\Norm{\nubold_{\F}^T \jump{\Hbbh} \nubold_{\F}}{\F}\\
& \overset{\eqref{bound:normal-normal}}{\lesssim}
    \hT^2 \Norm{\f-\fh}{\omegaF} 
    + \Norm{\Dtwo u - \Hbbh}{\omegaF} .
\end{split}
\end{equation}
The lower bound for the conforming part of the error is then
a consequence of~\eqref{lower-bound:c-bulk},
\eqref{bound:normal-normal},
and~\eqref{lower-bound:c-facet-1}.

\subsection{Upper and lower bounds of the nonconforming part of the error}
\label{subsection:upper-lower-bound-T2}
\paragraph*{The upper bound.}
We pick $\thetabf$ equal to $\psibf-\psibfh$,
$\psibfh$ as in Proposition~\ref{proposition:tools-nc},
and~$\thetabfI$ equal to the vector Cl\'ement quasi-interpolant
of~$\thetabf$ in $\Hbf^1(\Omega)$ as in~\eqref{Clement-qi}, and write
\[
\begin{split}
&   (\Hbbh, \symcurlbb \psibf)_{\Omega}
    \overset{\eqref{property-GH-2-bis}}{=}
    (\Hbbh, \symcurlbb (\psibf-\psibfh))_{\Omega}
    =: (\Hbbh, \symcurlbb \thetabf)_{\Omega}\\
& \overset{\eqref{property-GH-2-bis}}{=}  
    (\Hbbh, \symcurlbb (\thetabf-\thetabfI))_{\Omega}
    =  (\sym\Hbbh, \curlbb (\thetabf-\thetabfI))_{\Omega}\\
& \overset{\eqref{Clement-qi}}{\lesssim}
    \Big[ \sum_{\T\in\tauh} \hT^2 \Norm{\curlsym\Hbbh}{\T}^2
    + \sum_{\F\in\Fcaln} \hF \Norm{\tauboldF^T \jump{\sym\Hbbh}}{\F}^2
        \Big]^\frac12 \SemiNorm{\thetabf}{1,\Omega} \\
& \overset{\eqref{eq:tools-nc}}{\lesssim}     
    \Big[ \sum_{\T\in\tauh} \hT^2 \Norm{\curlsym\Hbbh}{\T}^2
    + \sum_{\F\in\Fcaln} \hF \Norm{\tauboldF^T \jump{\sym\Hbbh}}{\F}^2 \Big]^\frac12 
    \underbrace{\Norm{\symcurlbb \psibf}{\Omega}}_{=1}.
\end{split}
\]

\paragraph*{The lower bound.}
We begin with estimating the elemental terms.
For all~$\T$ in~$\tauh$, given $\btildeT$ as in \eqref{C0-bubble-bulk},
we have
\[
\begin{split}
\Norm{\curlsym \Hbbh}{\T}^2
& \lesssim (\curlsym \Hbbh,
        \underbrace{\btildeT \curlsym \Hbbh}_{=:\wbfT})_{\T} 
    = (\curl (\sym\Hbbh-\Dtwo u), \wbfT )_{\T} \\
& = (\sym \Hbbh-\Dtwo u, \curlbb \wbfT)_{\T}
  \le \Norm{\Dtwo u - \Hbbh}{\T} \Norm{\curlbb \wbfT}{\T}.
\end{split}
\]
By noting that
\[
\Norm{\curlbb \wbfT}{\T}
\lesssim \SemiNorm{\wbfT}{1,\T}
\overset{\eqref{C0-bubble-bulk-IE}}{\lesssim}
    \hT^{-1}  \Norm{\curlsym \Hbbh}{\T},
\]
we deduce, for all elements~$\T$,
\begin{equation} \label{lower-bound:nc-bulk}
\hT  \Norm{\curlsym \Hbbh}{\T}
\lesssim  \Norm{\Dtwo u - \Hbbh}{\T}.
\end{equation}
We now focus on the facet terms.
We focus on internal facets~$\F$
shared by the elements~$\To$ and~$\Ttw$;
the analysis for boundary facets is dealt with analogously.
Given $\Ecal$ and $\btildeF$ as in
\eqref{extension-operator} and \eqref{C0-bubble-facet},
we have
\[
\begin{split}
\Norm{\tauboldF^T \jump{\sym\Hbbh}}{\F}^2
& \overset{\eqref{C0-bubble-facet-IE}}{\lesssim}
    (\tauboldF^T \jump{\sym\Hbbh},
    \underbrace{\btildeF \Ecal(\tauboldF^T \jump{\sym\Hbbh})}_{=:\wbfF})_{\F} \\
& = \sum_{j=1}^2 \big[ (\curlsym \Hbbh, \wbfF)_{\Tj}
                        + (\sym \Hbbh, \curlbb \wbfF)_{\Tj} \big] \\
& = (\curlhsym \Hbbh, \wbfF)_{\omegatildeF}
                        + (\sym \Hbbh, \curlbb \wbfF)_{\omegatildeF}
  =: S_1 + S_2.
\end{split}
\]
We estimate the two terms on the right-hand side separately.
As for the term~$S_1$, we note that
\[
\begin{split}
S_1
& \lesssim  \hF \Norm{\curlhsym \Hbbh}{\omegaF} 
            \hF^{-1} \Norm{\wbfF}{\omegaF}  \\
& \overset{\eqref{C0-bubble-facet-IE}, \eqref{lower-bound:nc-bulk}}{\lesssim}
    \Norm{\Dtwo u - \Hbbh}{\T} \hF^{-\frac12} \Norm{\wbfF}{\F} 
  \le \Norm{\Dtwo u - \Hbbh}{\T} \hF^{-\frac12} \Norm{\tauboldF^T \jump{\sym\Hbbh}}{\F}.
\end{split}
\]
As for the term~$S_2$, using the fact that each column of~$\Dtwo u$ belongs
to~$\Hbf(\curlbb^{\zerobf},\Omega)$ and an integration by parts, we deduce
\[
\begin{split}
S_2
& = (\sym \Hbbh - \Dtwo u, \curlbb \wbfF)_{\omegaF}
  \le \Norm{\Dtwo u - \Hbbh}{\omegaF}  \Norm{\curlbb\wbfF}{\omegaF} \\
& \overset{\eqref{C0-bubble-facet-IE}}{\lesssim}
    \Norm{\Dtwo u - \Hbbh}{\omegaF} \hF^{-\frac12} \Norm{\wbfF}{\F} 
   \le \Norm{\Dtwo u - \Hbbh}{\omegaF} \hF^{-\frac12} \Norm{\tauboldF^T \jump{\sym\Hbbh}}{\F} .
\end{split}
\]
We combine the displays above and get
\begin{equation} \label{lower-bound:nc-facet}
\hF^{\frac12} \Norm{\tauboldF^T \jump{\sym\Hbbh}}{\F}
\lesssim \Norm{\Dtwo u - \Hbbh}{\omegaF}.
\end{equation}
The lower bound for the nonconforming part of the error
follows from~\eqref{lower-bound:nc-bulk}
and~\eqref{lower-bound:nc-facet}.

\subsection{Properties of the generalized Hessian}
\label{subsection:properties:GH}
Here, we prove
that the generalized Hessian in~\eqref{generalized:Hessian}
satisfies~\eqref{properties-GH:1} and~\eqref{properties-GH:2}.
Identity~\eqref{properties-GH:1}
is a consequence of the following computations:
for all $\wh$ in $\Vh\cap \H^2_0(\Omega)$,
\[
\begin{split}
(\Hbbh(\uh), \Dtwo \wh)_{\Omega}
\overset{\eqref{generalized:Hessian}}{=}
    (\Dtwoh\uh + \Lbbh(\uh), \Dtwo \wh)_{\Omega}
  \overset{\wh\in \H^2_0(\Omega),
        \eqref{IPDG-bilinear-form}}{=}
        \Bh(\uh,\wh)
    \overset{\wh\in\Vh, \eqref{method}}{=}
    (\f,\wh).
\end{split}
\]
The proof of identity~\eqref{properties-GH:2}
is more involved and is intrinsically related
to the proof of Green's identity~\eqref{Green-divdiv}.
We shall be equivalently proving that
\begin{equation} \label{equivalent-version-2nd-property}
(\Dtwoh \uh, \Abbh)_{\Omega}
= -(\Lbbh(\uh), \Abbh)_{\Omega},
\qquad\qquad\qquad
\forall \Abbh \in \underline{\mathbf P}_\p(\taun) \cap \Hbb(\divdivbfz,\Omega).
\end{equation}
As for the left-hand side of~\eqref{equivalent-version-2nd-property},
we integrate by parts elementwise twice and get
\[
\begin{split}
(\Dtwoh \uh, \Abbh)_{\Omega}
= \sum_{\T \in \tauh}
    \big[ (\uh, \divdivbf \Abbh)_{\T}
         + (\nabla\uh, \Abbh \nuboldT)_{\partial\T}
         - (\uh, \nuboldT^T (\divbf\Abbh))_{\partial\T}   \big].
\end{split}
\]
The bulk terms vanish since $\divdivbf\Abbh=0$ over~$\Omega$.
Consider the decomposition into tangential and normal parts
\begin{equation} \label{normal-tangential:splitting}
\Abbh \nuboldT
= (\nuboldT^T\Abbh\nuboldT)\nuboldT
  + (\tauboldT^T\Abbh\nuboldT)\tauboldT .
\end{equation}
This allows us to write
\small\[
(\Dtwoh \uh, \Abbh)_{\Omega}
= \sum_{\T \in \tauh}
    \big[ (\nuboldT^T \nabla\uh, \nuboldT^T\Abbh\nuboldT)_{\partial\T}
    + (\tauboldT^T \nabla\uh, \tauboldT^T\Abbh\nuboldT)_{\partial\T}
    - (\uh, \nuboldT^T (\divbf\Abbh))_{\partial\T}   \big] .
\]\normalsize
As for the second terms, we observe that
\[
(\tauboldT^T \nabla\uh, \tauboldT^T\Abbh\nuboldT)_{\partial\T}
= (\partialtbfOmega \uh, \tauboldT^T\Abbh\nuboldT)_{\partial\T}.
\]
With the notation as in~\eqref{Green-divdiv},
a tangential integration by parts on~$\partial\T$ gives
\[
(\tauboldT^T \nabla\uh, \tauboldT^T\Abbh\nuboldT)_{\partial\T}
= - (\uh, \partialtbfT (\tauboldT^T\Abbh\nuboldT))_{\partial\T}
   + \sum_{\F \in \Fcaln}  \sum_{\nu \in \VcalF}
    \sign_{\F,\nu} \uh(\nu) (\tauboldT^T \Abbh \nuboldT)(\nu) .
\]
Overall, we get
\begin{equation} \label{identities-Dtwouh}
\begin{split}
(\Dtwoh \uh, \Abbh)_{\Omega}
& = \sum_{\T \in \tauh}
    \big[ (\nuboldT^T \nabla\uh, \nuboldT^T\Abbh\nuboldT)_{\partial\T}
    - (\uh, \nuboldT^T \divbf\Abbh
    + \partialtbfT( \tauboldT^T \Abbh \nuboldT))_{\partial\T} \big] \\
& \quad + \sum_{\F \in \Fcaln}  \sum_{\nu \in \VcalF}
    \sign_{\F,\nu} \uh(\nu) (\tauboldF^T \Abbh \nuboldF)(\nu).
\end{split}
\end{equation}
We focus on the term on the
right-hand side of~\eqref{equivalent-version-2nd-property}.
The definition of the lifting operator in~\eqref{lifting}
and the splitting in~\eqref{normal-tangential:splitting} imply
\small\[
\begin{split}
& (\Lbbh(\uh), \Abbh)_{\Omega} \\
& = \Big(\jump{\uh}, 
    \average{\nuboldF^T \divbf \Abbh}\Big)_{\Fcaln}
  - \big(\jump{\nabla \uh}, 
    \average{\Abbh \nuboldF} \big)_{\Fcaln} \\
& = \Big( \jump{\uh}, 
    \average{\nuboldF^T \divbf \Abbh} \Big)_{\Fcaln}
  - \Big( \jump{\nuboldF^T \nabla \uh}, 
    \average{\nuboldF^T\Abbh \nuboldF} \Big)_{\Fcaln}
  - \Big( \jump{\tauboldF^T \nabla \uh}, 
    \average{\tauboldF^T\Abbh \nuboldF} \Big)_{\Fcaln}.
\end{split}
\] \normalsize
The continuity of~$\nuboldF^T\Abbh \nuboldF$ through facets,
cf. the $\Hbb(\divdivbf)$ conformity of $\Abbh$,
implies
\[
\begin{split}
 (\Lbbh(\uh), \Abbh)_{\Omega} 
& = \big( \jump{\uh}, 
    \average{\nuboldF^T \divbf \Abbh} \big)_{\Fcaln}
    - \big(\jump{\nuboldF^T \nabla \uh}, 
    \nuboldF^T\Abbh \nuboldF \big)_{\Fcaln} \\
& \quad  - \big( \jump{\tauboldF^T \nabla \uh}, 
    \average{\tauboldF^T\Abbh \nuboldF} \big)_{\Fcaln}.
\end{split}
\]
We rewrite the last term on the right-hand side,
also using an integration by parts on the boundaries
of the elements and standard DG manipulations, as
\[
\begin{split}
& - (\jump{\tauboldF^T \nabla \uh}, 
    \average{\tauboldF^T\Abbh \nuboldF})_{\Fcaln} 
= - (\jump{\partialtbfF \uh}, 
    \average{\tauboldF^T\Abbh \nuboldF})_{\Fcaln}\\
& = (\jump{\uh}, 
    \average{\partialtbfF  (\tauboldF^T\Abbh \nuboldF)})_{\Fcaln}
    - \sum_{\F \in \Fcaln}  \sum_{\nu \in \VcalF}
    \sign_{\F,\nu} \uh(\nu) (\tauboldF^T \Abbh \nuboldF)(\nu).
\end{split}
\]
Combining the above displays yields
\[
\begin{split}
(\Lbbh(\uh), \Abbh)_{\Omega}
& = (\jump{\uh}, 
    \average{\nuboldF^T \divbf \Abbh
            + \partialtbfF  (\tauboldF^T\Abbh \nuboldF)})_{\Fcaln} \\
& \quad 
    - (\jump{\nuboldF^T \nabla \uh}, 
    \nuboldF^T\Abbh \nuboldF)_{\Fcaln} 
    - \sum_{\F \in \Fcaln}  \sum_{\nu \in \VcalF}
    \sign_{\F,\nu} \uh(\nu) (\tauboldF^T \Abbh \nuboldF)(\nu).
\end{split}
\]
The continuity of~$\nuboldF^T \divbf \Abbh
+ \partialtbfF  (\tauboldF^T\Abbh \nuboldF)$ through facets,
cf. the $\Hbb(\divdivbf)$ conformity of~$\Abbh$,
implies
\begin{equation} \label{identities-Lcalh}
\begin{split}
-(\Lbbh(\uh),& \Abbh)_{\Omega}
= - (\jump{\uh}, 
    \nuboldF^T \divbf \Abbh
            + \partialtbfF  (\tauboldF^T\Abbh \nuboldF))_{\Fcaln} \\
& \quad 
    + (\jump{\nuboldF^T \nabla \uh}, 
    \nuboldF^T\Abbh \nuboldF)_{\Fcaln} 
    + \sum_{\F \in \Fcaln}  \sum_{\nu \in \VcalF}
    \sign_{\F,\nu} (\tauboldF^T \Abbh \nuboldF)(\nu) \uh(\nu).
\end{split}
\end{equation}
Up to rewriting the sum over the facets
as sums over the boundary of elements,
the right-hand sides in~\eqref{identities-Dtwouh} and~\eqref{identities-Lcalh}
coincide.
Identity~\eqref{equivalent-version-2nd-property}
and thus also~\eqref{properties-GH:2} follow.

\subsection{Comparison with the standard residual error estimator}
\label{subsection:comparison-standard-ee}
The aim of this section is to show that the
reliable and efficient error estimator
in~\eqref{global-error-estimator} based on $\Hbbhuh$
can be estimated from above by standard
DG residual error estimators available in the literature.

To see this, we recall the general order error estimator
in \cite[Definition 4.1]{Dong-Mascotto-Sutton:2021}
(\cite[eq. (4.1)]{Georgoulis-Houston-Virtanen:2011} is concerned
with for biharmonic problem in Laplace-Laplace formulation)\footnote{The first and sixth terms in display \eqref{estimator:DMS}
are written equivalently in \cite{Dong-Mascotto-Sutton:2021}
as (up to the scaling factors)
$\Norm{\f-\Delta^2 \uh}{\T}$
since for smooth functions $\Delta^2 \cdot = \divdivbf\Dtwo \cdot$
and $\Norm{\jump{\nablah \Delta^2_{\mathrm \h} \uh} \cdot \nuboldF}{\F}$
since $\nablah\Delta_{\mathrm \h}^2\cdot = \divbf_{\mathrm \h} \Dtwoh$.}:
\begin{subequations}
\label{DG-error-estimator}
\begin{equation} \label{estimator:DMS}
\begin{split}
\gimelT ^2
:= & \sum_{\ell=1}^6 \gimel_{\ell,\T}^2
    := \hT^4 \Norm{\f-\divdivbf \Dtwo \uh}{\T}^2
   + \sum_{\F\in\FcalT} \Big(
   \hF^{-3} \Norm{\jump{\uh}}{\F}^2
    + \hF^{-1} \Norm{\jump{\nabla\uh}}{\F}^2\\
&   + \hF \Norm{\jump{\Dtwoh\uh}\nuboldF}{\F}^2
    + \hF \Norm{\jump{\Dtwoh\uh}\tauboldF}{\F}^2
    + \hF^{3} \Norm{\jump{\divbf{_\mathrm\h} \Dtwoh \uh}\cdot \nuboldF}{\F}^2
        \Big).
\end{split}
\end{equation}
For future convenience,
we also define the corresponding global error estimator
\begin{equation} \label{global-DG-error-estimator}
    \gimel^2 := \sum_{\T\in\tauh} \gimelT^2.
\end{equation}
\end{subequations}
We want to prove that $\etaT \lesssim \gimelT$;
to this aim, we estimate each $\eta_{j,\T}$, $j=1,\dots,5$,
in~\eqref{local-error-estimator} in terms of combinations
of $\gimel_{\ell,\T}$, $\ell=1,\dots,6$,
in~\eqref{estimator:DMS}.
In all cases, we split the generalized Hessian~$\Hbbhuh$
into the standard Hessian $\Dtwoh\uh$
and the lifting $\Lbbh(\uh)$.
As for the latter, using inverse estimates and the stability
result~\eqref{lifting} bounds all the contributions
involving $\Lbbh(\uh)$ in terms
of~$\gimel_{2,\T}$ and~$\gimel_{3,\T}$.
Therefore, we only focus on the remainder parts 
involving the standard broken Hessian.

The resulting term stemming from~$\eta_{1,\T}$
boils down to $\gimel_{1,\T}$.
On the other hand,
as for the term resulting from~$\eta_{2,\T}$, we have
\[
\sum_{\F\in\FcalT} \hF \Norm{\nuboldF^T \jump{\Dtwoh\uh}\nuboldF}{\F}^2
\le \gimel_{4,\T}^2.
\]
As for the term resulting from~$\eta_{3,\T}$, we write
\[
\begin{split}
& \sum_{\F\in\FcalT} 
    \hF^3 \Norm{\partialtbfF(\tauboldF^T \jump{\Dtwoh\uh}\nuboldF)
    + \nuboldF^T \jump{\divbf\Dtwoh\uh}}{\F}^2 \\
& \lesssim \sum_{\F\in\FcalT}  \hF^3
    \big(\Norm{\partialtbfF(\tauboldF^T \jump{\Dtwoh\uh}\nuboldF)}{\F}^2
    + \Norm{\nuboldF^T \jump{\divbf\Dtwoh\uh}}{\F}^2 \big) \\
& \lesssim \sum_{\F\in\FcalT}  \hF
    \Norm{\jump{\Dtwoh\uh}\nuboldF}{\F}^2
    + \sum_{\F\in\FcalT} \hF^3 \Norm{\nuboldF^T \jump{\divbf\Dtwoh\uh}}{\F}^2
    = \gimel_{4,\T}^2 + \gimel_{6,\T}^2.
\end{split}
\]
The term resulting from~$\eta_{4,\T}$ vanishes
as $\curlbf \Dtwoh\uh$ is piecewise zero.
As for the term resulting from~$\eta_{5,\T}$, we write
\[
\sum_{\F\in\FcalT} \hF \Norm{\tauboldF^T  \jump{\Dtwoh\uh}}{\F}^2
= \gimel_{5,\T}^2.
\]
The considerations above imply that the local error estimator~\eqref{local-error-estimator}
can be estimated from above by
the standard DG residual error estimator.
In Section~\ref{section:nr} below,
we shall assess the practical performance
of the two error estimators.

\section{Numerical results} \label{section:nr}

\paragraph*{Test cases.}
The first test case,
cf. with \cite[pp. 107--108]{Grisvard:1992},
is set on the L-shaped domain
$\Omega_1:= (-1,1)^2 \setminus ( [0,1) \times (-1,0] )$
with exact solution
\begin{equation} \label{u1}
u_1(r,\theta) = r^{1+z} g(\theta),
\end{equation}
where $(r,\theta)$ are the polar coordinates at the
re-entrant corner $(0,0)$,
$z=0.544483736782464$,
and, for $\omega=3\pi/2$,
\[
\begin{split}
g(\theta)
& := \Big[ \frac{1}{1-z} \sin( (z-1) \omega)
            - \frac{1}{z+1} \sin( (z+1)\omega) \Big]
            \cos( (z-1)\theta) - \cos( (z+1)\theta)\\
&   \qquad -\Big[ \frac{1}{1-z} \sin( (z-1) \theta)
            - \frac{1}{z+1} \sin( (z+1)\theta) \Big]
            \cos( (z-1)\omega) - \cos( (z+1)\omega) .
\end{split}
\]
The function~$u_1$ has inhomogeneous clamped boundary
conditions on part of~$\partial \Omega_1$
(it has effectively homogeneous boundary
conditions on the two edges at the re-entrant corner);
we refer to the next paragraph for further discussions
on this aspect.

The second test case is set on the unit square domain $\Omega_2=(0,1)^2$ with exact solution
\begin{equation} \label{u2}
    u_2(x,y) = \sin^2(\pi \ x) \sin^2(\pi \ y).
\end{equation}

\paragraph*{Handling inhomogeneous boundary conditions.}
The analysis discussed so far considered
the biharmonic problem with homogeneous 
clamped boundary conditions.
The inhomogeneous case follows closely the presentation
of the previous sections.
For the sake of completeness, we sketch
here some details on the inhomogeneous boundary conditions
\[
u=g_1,
\qquad\qquad\qquad\qquad\qquad
\partialnbfOmega u = g_2
\qquad\qquad\qquad \text{on } \partial\Omega.
\]
Formal computations reveal that
\begin{equation} \label{formal-computations}
\begin{split}
\taubfOmega^T \Dtwo u_{|\partial\Omega}
&   = \partialtbfOmega (\nabla u_{|\partial\Omega} )
    = \partialtbfOmega ( (\partialtbfOmega g_1)\tauboldOmega
                        + g_2 \nuboldOmega)
    = (\partialtwotbfOmegatwo g_1)\tauboldOmega
            + (\partialtbfOmega g_2) \nuboldOmega,\\
\nabla u_{|\partial\Omega}
&   = (\nuboldOmega^T \nabla u_{|\partial\Omega}) \nuboldF
        + (\tauboldOmega^T \nabla u_{|\partial\Omega}) \tauboldOmega
    =  g_2 \nuboldOmega + (\partialtbfOmega g_1)\tauboldOmega. \\
\end{split}
\end{equation}
For the definition of the lifting
operators~\eqref{lifting},
the jump operators on boundary facets
modify into
\[
\jump{\uh}_{\F} = \uh{}_{|\F} - g_1,
\qquad\qquad\qquad\qquad
\jump{\nabla \uh} _\F
\overset{\eqref{formal-computations}}{=}
\nabla \uh{}_{|\F} - [g_2 \nuboldOmega + (\partialtbfOmega g_1)\tauboldOmega].
\]
Identity~\eqref{properties-GH:2} holds true only
up to some terms involving the boundary data
that eventually modifies
the local error estimator in~\eqref{local-error-estimator},
and notably the term $\eta_{5,\T}$, which
on boundary facets~$\F$ in $\EcalT$ now reads
\begin{equation} \label{jump:sym-GH}
\tauboldOmega^T \jump{\sym\Hbbhuh}_{\F}
= (\taubfOmega^T \sym\Hbbhuh 
    - (\partialtwotbfOmegatwo g_1)\tauboldOmega
    - (\partialtbfOmega g_2) \nuboldOmega)_{|\F} .
\end{equation}
The reason for this resides in the fact that
identity~\eqref{T2-as-sup}
is not true any more.
In particular, we have
\[
\begin{split}
& (\Dtwo u, \symcurlbb \thetabf)_{\Omega}
    = (\Dtwo u, \curlbb \thetabf)_{\Omega}
    = (\nabla u, \curlbb\thetabf \cdot
                \nuboldOmega)_{\partial\Omega}
    = (\nabla u, \nabla\thetabf \cdot
                \tauboldOmega)_{\partial\Omega} \\
&   \overset{\eqref{formal-computations}}{=}
    ((\partialtbfOmega g_1)\tauboldOmega
    + g_2 \nuboldOmega ,
        \partialtbfOmega \thetabf)_{\partial\Omega}
    = - ((\partialtwotbfOmegatwo g_1) \tauboldOmega, \thetabf)_{\partial\Omega}
        - ( (\partialtbfOmega g_2) \nuboldOmega,
        \thetabf)_{\partial\Omega},
\end{split}
\]
which is different from zero in general
(while it was zero for homogeneous
boundary conditions).
In fact, the discrete counterpart
of the above identities holds,
whence~\eqref{properties-GH:2}
is not valid in general,
but should involve the boundary data
on the right-hand side
with an arbitrary test function $\thetabfh$.
This reasoning eventually combines with the estimates
in Section~\ref{subsection:upper-lower-bound-T2}
and yields~\eqref{jump:sym-GH}.
    
\paragraph*{Outline of the remainder of the section.}
In Section~\ref{subsection:nr-2error-measures},
we compare the performance of the error measure
on the left-hand side of~\eqref{rewriting-in-terms-of-max}
with the standard DG error under uniform refinements
and for different orders.
We investigate the practical efficiency of the error
estimators in~\eqref{global-error-estimator}
and the standard DG-residual error estimator
under uniform $\h$- and $\p$-refinements
in Section~\ref{subsection:nr-effectivity-h}.
Section~\ref{subsection:nr-adaptive} is devoted
to introduce an adaptive algorithm driven
by the error estimator in~\eqref{global-error-estimator}
and investigate its performance.
In~\eqref{IPDG-bilinear-form}, we shall always consider
\[
\csigma = 3 \p^6,
\qquad\qquad\qquad
\ctau   = 9 \p^2,
\]
which mimics the $\p$-scaling in~\cite{Dong-Mascotto-Sutton:2021}.
Moreover, in the local estimators
\eqref{error-estimator} and~\eqref{DG-error-estimator},
we replace $\h^{\lambda}$,
$\lambda$ in $\Z$,
with $\h^{\lambda}/\p^{\lambda}$;
this poses no modifications
in the analysis and gives a
scaling in terms of the degree of the scheme
that recalls that in~\cite{ChaumontFrelet:2025}.
Throughout, we consider triangular meshes.
The IPDG solution in~\eqref{method}
is computed using a direct solver.

\subsection{Comparison of two error measures}
\label{subsection:nr-2error-measures}
Given $u$ and $\uh$ the solutions to~\eqref{problem:weak}
and~\eqref{method},
$\Hbbhuh$ as in~\eqref{generalized:Hessian},
and $\Norm{\cdot}{\DG}$ as in~\eqref{DG-norm},
we are interested here in comparing the following two
error measures:
\begin{equation} \label{error-measures}
\Norm{\Dtwo u - \Hbbhuh}{\Omega},
\qquad\qquad\qquad\qquad
\Norm{u-\uh}{\DG}.
\end{equation}
We pick the exact solution~$u_2$ in~\eqref{u2},
$\p=2$ and $\p=5$,
sequences of uniformly refined meshes,
and report the results in Tables~\ref{tab:p-2}
and~\ref{tab:p-5}.

\begin{table}[htp]
  \begin{center}
  \begin{tabular}{|r|lr|lr|}
    \hline
    level & $\Norm{u-\uh}{\DG}$ & order & $\Norm{\Dtwo u - \Hbbhuh}{\Omega}$ & order \\ \hline
    $1$ & $1.5095\times 10^{1}$ & -- & $9.4596\times 10^{0}$ & -- \\
    $2$ & $1.4193\times 10^{1}$ & $0.09$ & $1.1561\times 10^{1}$ & $-0.29$ \\
    $3$ & $1.1255\times 10^{1}$ & $0.33$ & $1.0065\times 10^{1}$ & $0.20$ \\
    $4$ & $4.7746\times 10^{0}$ & $1.24$ & $5.3481\times 10^{0}$ & $0.91$ \\
    $5$ & $2.4031\times 10^{0}$ & $0.99$ & $2.7127\times 10^{0}$ & $0.98$ \\
    $6$ & $1.2006\times 10^{0}$ & $1.00$ & $1.3597\times 10^{0}$ & $1.00$ \\
    $7$ & $5.9949\times 10^{-1}$ & $1.00$ & $6.8016\times 10^{-1}$ & $1.00$ \\
    $8$ & $2.9948\times 10^{-1}$ & $1.00$ & $3.4011\times 10^{-1}$ & $1.00$ \\
    $9$ & $1.4966\times 10^{-1}$ & $1.00$ & $1.7006\times 10^{-1}$ & $1.00$ \\
    $10$ & $7.4812\times 10^{-2}$ & $1.00$ & $8.5029\times 10^{-2}$ & $1.00$ \\
    $11$ & $3.7401\times 10^{-2}$ & $1.00$ & $4.2514\times 10^{-2}$ & $1.00$ \\
    \hline
  \end{tabular}
  \end{center}
  \caption{Error measures in~\eqref{error-measures}.
  Exact solution~$u_2$ in~\eqref{u2};
    $\p=2$; sequences of uniformly refined meshes.}
  \label{tab:p-2}
\end{table}
\begin{table}[htp]
  \begin{center}
  \begin{tabular}{|r|lr|lr|}
    \hline
    level & $\Norm{u-\uh}{\DG}$ & order & $\Norm{\Dtwo u - \Hbbhuh}{\Omega}$ & order \\ \hline
    $1$ & $9.1635\times 10^{0}$ & -- & $8.7192\times 10^{0}$ & -- \\
    $2$ & $1.9076\times 10^{0}$ & $2.26$ & $1.7656\times 10^{0}$ & $2.30$ \\
    $3$ & $2.8588\times 10^{-2}$ & $6.06$ & $2.5638\times 10^{-2}$ & $6.11$ \\
    $4$ & $7.0726\times 10^{-3}$ & $2.02$ & $6.6088\times 10^{-3}$ & $1.96$ \\
    $5$ & $4.4821\times 10^{-4}$ & $3.98$ & $4.1799\times 10^{-4}$ & $3.98$ \\
    $6$ & $2.8049\times 10^{-5}$ & $4.00$ & $2.6155\times 10^{-5}$ & $4.00$ \\
    $7$ & $1.7624\times 10^{-6}$ & $3.99$ & $1.6442\times 10^{-6}$ & $3.99$ \\
    \hline
  \end{tabular}
  \end{center}
  \caption{Error measures in~\eqref{error-measures}.
  Exact solution~$u_2$ in~\eqref{u2};
    $\p=5$; sequences of uniformly refined meshes.}
  \label{tab:p-5}
\end{table}

The two error measures in~\eqref{error-measures}
converge with the expected order
and deliver approximately the same accuracy.

\subsection{Efficiency under uniform refinements}
\label{subsection:nr-effectivity-h}
We assess
the performance of the effectivity indices
\begin{equation} \label{effectivity-indices}
\FfrakHbbh := \frac{\eta}{\Norm{\Dtwo u - \Hbbhuh}{\Omega}},
\qquad\qquad\qquad
\FfrakDG := \frac{\gimel}{\Norm{u - \uh}{\DG}},
\end{equation}
where $\eta$ and~$\gimel$ are given
in~\eqref{global-error-estimator} 
and~\eqref{global-DG-error-estimator}.
First, we consider uniform mesh refinements,
pick the exact solution $u_1$ in~\eqref{u1},
$\p=2,3,4,5$, and~$\p=6,7,8,9$
and report the results
in Figures~\ref{figure:effectivity-h-version}
and~\ref{figure:effectivity-h-version-5-9}.

\begin{figure}[htb]
\centering
\includegraphics[width=0.45\textwidth]{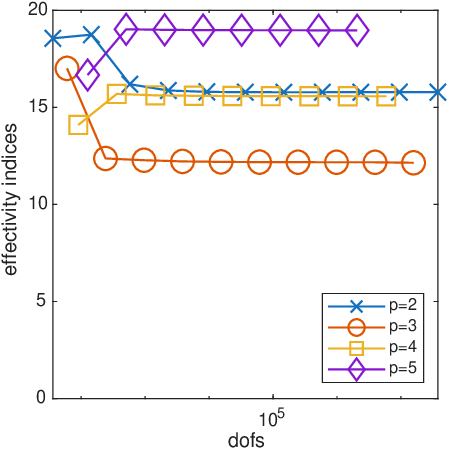}
\qquad
\includegraphics[width=0.45\textwidth]{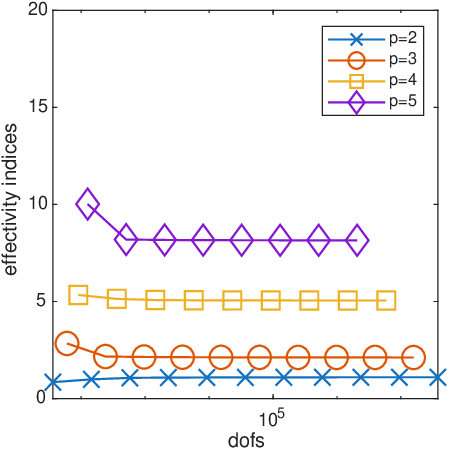}
\caption{Effectivity indices in~\eqref{effectivity-indices}.
Test case with exact solution~$u_1$ in~\eqref{u1};
$\p=2,3,4,5$; uniform mesh refinements.
\emph{Left-panel}: effectivity index~$\FfrakHbbh$.
\emph{Right-panel}: effectivity index~$\FfrakDG$.}
\label{figure:effectivity-h-version}
\end{figure}

\begin{figure}[htb]
\centering
\includegraphics[width=0.45\textwidth]{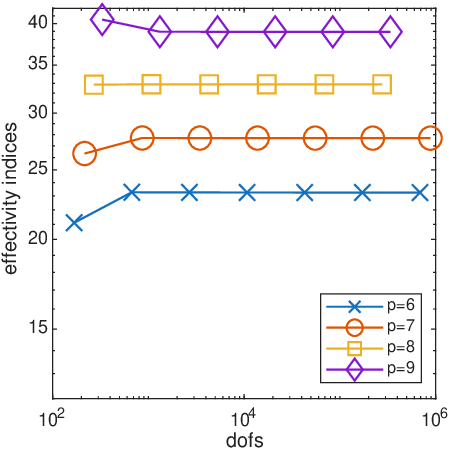}
\qquad
\includegraphics[width=0.45\textwidth]{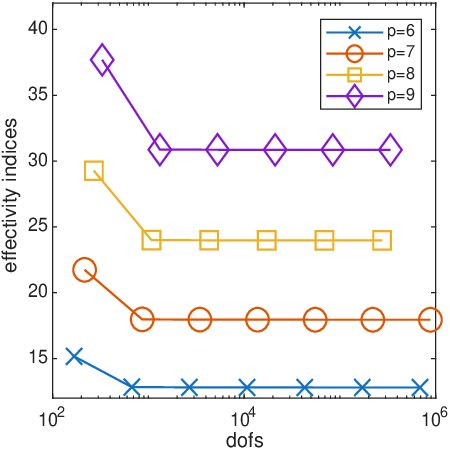}
\caption{Effectivity indices in~\eqref{effectivity-indices}.
Test case with exact solution~$u_1$ in~\eqref{u1};
$\p=6,7,8,9$; uniform mesh refinements.
\emph{Left-panel}: effectivity index~$\FfrakHbbh$.
\emph{Right-panel}: effectivity index~$\FfrakDG$.}
\label{figure:effectivity-h-version-5-9}
\end{figure}

Next, for the test case with exact
solution~$u_1$ in~\eqref{u1}
and a fixed structured mesh of~$6$ triangles,
we investigate the performance
of the effectivity indices in~\eqref{effectivity-indices}
under $\p$-refinements in Figure~\ref{figure:effectivity-p-version}
for $\p=2,\dots,20$.

\begin{figure}[htb]
\centering
\includegraphics[width=0.45\textwidth]{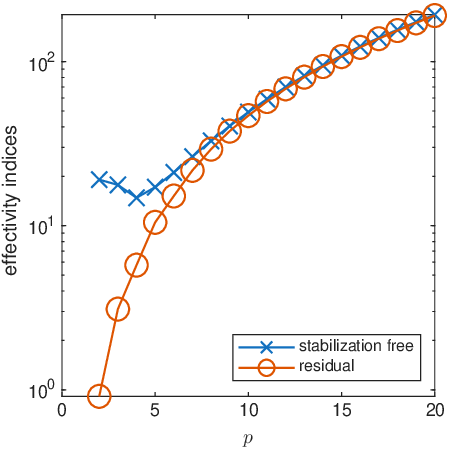}
\caption{Effectivity indices in~\eqref{effectivity-indices}.
Test case with exact solution~$u_1$ in~\eqref{u1};
fixed structured mesh of~$6$ triangles;
$\p$-refinement.
Effectivity indices~$\FfrakHbbh$ and~$\FfrakDG$
in~\eqref{effectivity-indices}.}
\label{figure:effectivity-p-version}
\end{figure}

From Figures~\ref{figure:effectivity-h-version}
and~\ref{figure:effectivity-h-version-5-9},
we assess the efficiency of the error estimator
in~\eqref{error-estimator};
for low degree~$\p$, the standard residual error estimator
exhibits lower indices,
while for larger values of~$\p$
the indices tend to be more similar.
From Figure~\ref{figure:effectivity-p-version},
we observe that the two error estimators
are not $\p$-robust
(as one could have expected from the
use of inverse estimates in the derivation
of the lower bound~\eqref{apos-bound-2});
importantly, for increasing~$\p$,
the effectivity indices converge to one another.

\subsection{Numerical results for an adaptive algorithm}
\label{subsection:nr-adaptive}
We use here the error estimators in~\eqref{error-estimator}
and~\eqref{DG-error-estimator}
to drive an adaptive algorithm of the form
\[
\textbf{SOLVE}
\qquad\longrightarrow\qquad
\textbf{ESTIMATE}
\qquad\longrightarrow\qquad
\textbf{MARK}
\qquad\longrightarrow\qquad
\textbf{REFINE}
\]
where
\begin{itemize}
\item in the \textbf{SOLVE} step, we employ a direct solver;
\item in the \textbf{ESTIMATE} step,
we employ the error estimators in~\eqref{error-estimator}
and~\eqref{DG-error-estimator}
\item in the \textbf{MARK} step, we use D\"orfler's marking
strategy with marking parameter $1/2$;
\item in the \textbf{REFINE} step,
we use newest vertex bisection
(with the closure algorithm).
\end{itemize}
We run the adaptive algorithm on the test case
with exact solution~$u_1$ in~\eqref{u1}
for fixed $\p=2$ and~$5$.
In Figure~\ref{figure:adaptive},
we report the error measures in~\eqref{error-measures},
and the error estimators in~\eqref{error-estimator}
and~\eqref{DG-error-estimator}
under uniform and adaptive mesh refinements.

\begin{figure}[htb]
\centering
\includegraphics[width=0.45\textwidth]{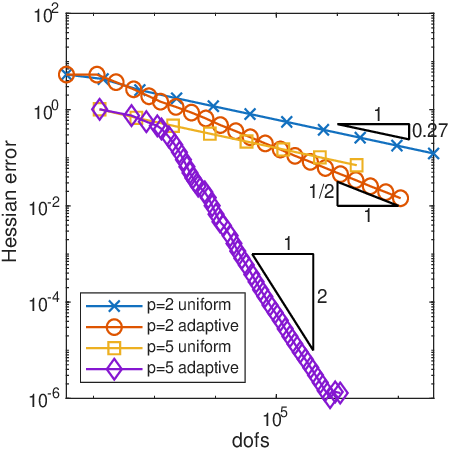}
\qquad
\includegraphics[width=0.45\textwidth]{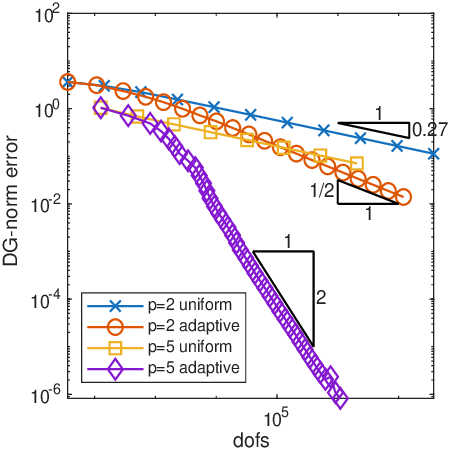}
\caption{Test case with exact solution~$u_1$ in~\eqref{u1};
$\p=2,5$; uniform and adaptive mesh refinements.
\emph{Left-panel}: $\Norm{\Dtwo u - \Hbbhuh}{\Omega}$;
the adaptive algorithm is driven by the error estimator
in~\eqref{error-estimator}.
\emph{Right-panel}: $\Norm{u - \uh}{\DG}$;
the adaptive algorithm is driven by the error estimator
in~\eqref{DG-error-estimator}.}
\label{figure:adaptive}
\end{figure}

Optimal convergence using the adaptive algorithm
with both estimators is recovered;
the expected suboptimal rates are observed
for under uniform mesh refinements.

\section{Conclusions} \label{section:conclusions}
\paragraph*{Contents of this work.}
We designed an error estimator for an IPDG discretization
of the biharmonic problem,
which does not depend explicitly on the stabilization of the
method. Upper and lower bounds with respect to a suitable
error measure are proven.
Crucial tools in the analysis were
a rewriting of the error measure based on the
two dimensional div-div complex,
quasi-interpolants in $\mathcal C^1$
finite element spaces,
and a Korn-type inequality for the $\symcurlbb$ operator.
The reliability and efficiency
of the error estimator were
verified in practice as well.
An adaptive algorithm was finally used
to construct sequences of geometrically refined
meshes to approximate solutions with corner singularities.

\paragraph{Outlook.}
Open problems that may be tackled in the future
include the convergence analysis of the adaptive algorithm,
which may be simpler than that
for the standard IPDG error estimator
due to the absence of explicit stabilizations in
the proposed error estimator; cf. also Remark~\ref{remark:powers-h}.
The estimates presented in this work can be also made explicit
in terms of the degree of accuracy;
cf. Remark~\ref{remark:p-dependence-bounds}.
Besides, differently from the standard a posteriori error estimates
as in~\cite{Georgoulis-Houston-Virtanen:2011, Brenner-Gudi-Sung:2010},
where the 2D and 3D cases have a similar treatment,
we have here crucial differences that deserve a separate contribution:
the Green's identity~\eqref{Green-divdiv}
gets more complicated, in particular the ``vertex terms''
are substituted by ``edge'' terms that have to be carefully
handled in the a posteriori error analysis;
considering nontrivial topologies
gest more complicated compared to the 2D case
essentially because in 3D the $\curlbf$
operator is not the rotated $\nabla$.

\paragraph*{Acknowledgments.}
The authors are grateful to Johnny Guzm\'an
for discussions concerning the topics
in Appendices~\ref{appendix:nontrivial-topologies}
and~\ref{appendix:tools-nonconforming}.
LM has been partially funded by MUR (PRIN2022 research grant n. 202292JW3F).
LM was also partially supported by the European Union (ERC Synergy, NEMESIS, project number 101115663).
Views and opinions expressed are however those of the author(s) only and do not necessarily reflect
those of the European Union or the European Research Council Executive Agency. 
LM is member of the Gruppo Nazionale Calcolo Scientifico-Istituto Nazionale di Alta Matematica (GNCS-INdAM).

{\footnotesize
\bibliography{bibliogr}
}
\bibliographystyle{plain}

\appendix

\section{Proof of Lemma~\ref{lemma:special-split}}
\label{appendix:nontrivial-topologies}

Recall the following result,
cf. \cite[Lemma 2.1.2]{Bernardi-Girault-Hecht-Raviart-Riviere:2025}.
\begin{lemma} \label{lemma:Girault-Raviart}
Given~$\Omega$ as in~\eqref{Omega},
the following statements are equivalent:
for all $\xiboldtilde$ in $\Hbf(\divz,\Omega)$,
\begin{enumerate}
\item there exists $\theta$ in $\H(\curlbf,\Omega)$
such that $\xiboldtilde=\curlbf\theta$;
\item $\int_{\Sigmaj} \xiboldtilde\cdot \nuboldOmega=0$ for all $j=1,\dots,N$
\footnote{In~\cite{Bernardi-Girault-Hecht-Raviart-Riviere:2025},
the index~$j$ starts from~$0$, which is equivalent to what
we ask here due to
the fact that $\xiboldtilde$ is divergence free
and the corresponding compatibility conditions hold true.}.
\end{enumerate}
\end{lemma} 

\begin{remark} \label{remark:construction-cohomology}
It is possible to construct
$\sigmaboldj$ in $\Hbf(\divz,\Omega)$
such that
\[
\int_{\Sigmaj} \sigmaboldi \cdot \nuboldOmega = \delta_{i,j}
\qquad\qquad \forall i,j = 1,\dots,N.
\]
For instance, we take $\sigmaboldj:=\nabla u_j$,
where each~$u_j$ in $\H^1(\Omega)\cap\L^2_0(\Omega)$ solves
\[
\begin{cases}
\Delta u_j = 0
    & \text{in } \Omega\\
\nabla u_j \cdot \nuboldOmega = \delta_{i,j}
    & \text{on } \Sigma_i, \; \forall i=1,\dots,N\\
\nabla u_j \cdot \nubold_0 = -\vert\Sigma_j\vert / \vert\Sigma_0\vert
    & \text{on } \Sigma_0
\end{cases}
\qquad\qquad
\forall j=1,\dots,N
\]
Standard elliptic regularity theory in 2D entails that
each~$u_j$ belongs to $\H^{\frac32+\varepsilon}(\Omega)$.
We introduce the lowest order Raviart-Thomas
interpolant~$\sigmabold_{\h,j}$ of~$\sigmaboldj$,
which preserves the normal traces of~$\sigmaboldj$
on each~$\Sigmaj$ and is divergence free.

Let~$\psij$ be any function
in $C^\infty(\overline\Omega)$
with $\psij{}_{|\Sigma_i}=\delta_{i,j}$
for all $i,j=1,\dots,N$;
for instance, we can mollify piecewise constant
functions taking values zero and one
near the corresponding~$\Sigma_i$ and~$\Sigmaj$, respectively.
With such functions fixed, we write
\begin{equation} \label{eq:normal-bdd--1}
\int_{\Sigmaj}
\xiboldtilde \cdot \nuboldOmega
= \int_{\Omega}
\xiboldtilde \cdot \nabla  \psij 
\le \Norm{\xiboldtilde}{\Hbfostar} \Norm{\nabla \psij}{1,\Omega}       
\lesssim \Norm{\xiboldtilde}{\Hbfostar}
\qquad
\forall  \xiboldtilde \in \Hbf(\divz,\Omega),
\end{equation}
where the hidden constant only depends on the choice
of the~$\psij$ and thus on~$\Omega$.
\eremk
\end{remark}

Based on this construction, we prove a technical result. Define
\[
\Hbfstardivz
:= \{ \vbf\in \Hbfostar \mid \div\vbf=0 \}.
\]

\begin{lemma} \label{lemma:special-split-auxiliary}
For all $\xibold$ in $\Hbf(\divz,\Omega)$,
there exist~$\sigmaboldh$
in $\RT_0(\taun)\cap \Hbf(\divz,\Omega)$
and $\theta$ in $\H^1(\Omega)$ such that
\begin{equation} \label{eq:special-split}
\xibold = \curlbf\theta + \sigmaboldh.
\end{equation}
Moreover, for all $\xibold$ in $\Hbfstardivz$,
there exist~$\sigmaboldh$ in $\RT_0(\taun)\cap \Hbf(\divz,\Omega)$
and~$\theta$ in~$\L^2(\Omega)$ such that
\begin{equation} \label{eq:special-split-bis}
\xibold = \curlbf\theta + \sigmaboldh.
\end{equation}
\end{lemma}
\begin{proof}
Consider the notation as in
Remark~\ref{remark:construction-cohomology}.\\
\textbf{Proof of~\eqref{eq:special-split}.}
Given~$\xibold$ in $\Hbf(\divz,\Omega)$,
we set
\begin{equation} \label{def:xiboldtilde}
\xiboldtilde
:= \xibold - \sigmaboldh
:= \xibold -
    \sum_{j=1}^N  \Big( \int_{\Sigmaj} \xibold\cdot\nuboldOmega \Big)
    \sigmabold_{\h,j},
    \qquad \text{where } \sigmaboldh \in \RT_0(\taun).
\end{equation}
By construction, $\int_{\Sigmaj} \xiboldtilde\cdot\nuboldOmega = 0$
for all $j=1,\dots,N$ and~$\xiboldtilde$
is in $\Hbf(\divz,\Omega)$.
Therefore, we are in a position to apply Lemma~\ref{lemma:Girault-Raviart}:
there exists~$\theta$ in~$\H^1(\Omega)$ such that
\begin{equation} \label{xibold-expansion}
\xiboldtilde = \curlbf \theta,
\quad\theta\in \H^1(\Omega)
\qquad\overset{\eqref{def:xiboldtilde}}{\Longrightarrow}\qquad
\xibold = \curlbf \theta + \sigmaboldh,
\end{equation}
which is~\eqref{eq:special-split}.
For later use, we observe that
\small\begin{equation} \label{sigma-part}
\Norm{\sigmaboldh}{\Hbfostar}
    + \Norm{\sigmaboldh}{}
    + \Norm{\curlbf\theta}{\Hbfostar}
\overset{\eqref{xibold-expansion}}{\lesssim}
    \Norm{\sigmaboldh}{\Hbfostar}
    + \Norm{\sigmaboldh}{}
    + \Norm{\xibold}{\Hbfostar}
\overset{\eqref{eq:normal-bdd--1}}{\lesssim}
    \Norm{\xibold}{\Hbfostar},
\end{equation} \normalsize
where the hidden constants depend on~$\Omega$.\\
\textbf{Density of~$\Hbf(\divz,\Omega)$
in~$\Hbfstardivz$.}
Using~\cite[Corollary~4.8 (b)]{Costabel-McIntosh:2010},
given $\xibold$ in $\Hbfstardivz$,
we have the decomposition
\[
\xibold=\curlbf\theta + \sigmabold,
\qquad\qquad\qquad
\theta \in \L^2(\Omega),
\;
\sigmabold \in \mathbf C^\infty(\Omega).
\]
As~$\xibold$ is divergence free, also $\sigmabold$ is divergence free.
We consider the extension by zero
of~$\theta$ to the whole~$\Rbb^2$,
which we still call $\theta$.
We mollify~$\theta$ and get a smooth~$\rhodelta\star\theta$.
Since the mollification commutes with differential operators in $\Rbb^d$, we have
\[
\xibolddelta
:=\curlbf(\rhodelta\star\theta)+\sigmabold
=\rhodelta\star(\curlbf\theta)+\sigmabold
    \in \Hbf(\divz,\Omega).
\]
We have that~$\thetan-\theta$ converges to~$0$ in $\L^2(\Omega)$
by standard approximation properties of mollifiers.
Since $\curlbf$ is continuous from $\L^2(\Omega)$
to $\Hbfstardivz$, we deduce
that $\curlbf(\thetan-\theta)$ converges to~$\zerobf$
in $\Hbfstardivz$.
Since $\div\xibold = 0 = \div \sigmabold$, we get
\[
\xibolddelta \to \xibold
\text{ in } \Hbfstardivz.
\]

\noindent \textbf{Proof of~\eqref{eq:special-split-bis}.}
Take a sequence $\{\xiboldn\}$ in~$\Hbf(\divz,\Omega)$
converging to~$\xibold$ in $\Hbfstardivz$.
Using~\eqref{eq:special-split},
we have the sequence of regular decompositions
\[
\xiboldn =
\curlbf \thetan + \sigmaboldhn,
\qquad\qquad\qquad
\thetan \in \H^1(\Omega),\;
\sigmaboldhn \in \RT_0(\taun) \cap \Hbf(\divz,\Omega).
\]
Using bound~\eqref{sigma-part} on each~$\xibold_n$,
we deduce that the sequence of the~$\sigmabold_{\h,n}$
is uniformly bounded in~$\Lbf^2(\Omega)$;
thus, it converges to a vector field~$\sigmaboldh$
in $\RT_0(\taun) \cap \Hbf(\divz,\Omega)$ up to subsequences.
We deduce that~$\curlbf \thetan$ also admits a subsequence
converging to~$\Bbf$ in $\Hbfstardivz$.
On the other hand, in 2D the $\curlbf$ operator is the rotated~$\nabla$ operator.
Moreover, the image of~$\L^2(\Omega)$ through~$\nabla$ is closed
in $\Hbfostar$, as a consequence of Peetre-Tartar's lemma
\cite[Theorem 1.2.1]{Bernardi-Girault-Hecht-Raviart-Riviere:2025}
and the Ne\v cas-Lions inequality \cite[Lemma 3.7.1]{Necas:1967}.
Hence, there exists~$\theta$ in $\L^2(\Omega)$
such that~$\curlbf\theta=\Bbf$.

All in all, we proved that~$\xiboldn$ converges
to $\curlbf \theta + \sigmaboldh$ in $\Hbfstardivz$
for $\theta$ in $\L^2(\Omega)$
and $\sigmaboldh$ in $\RT_0(\taun) \cap \Hbf(\divz,\Omega)$,
which is the assertion.
\end{proof}
We are now in a position
to prove Lemma~\ref{lemma:special-split}.
\begin{proof}
Define $\Bbf := \divbf \Abb$, which belongs to~$\Hbfostar$.
Since $\div \Bbf = 0$, \eqref{eq:special-split-bis} entails
\[
\Bbf = \curlbf \Phi + \sigmaboldh,
\qquad\qquad \text{for }
\Phi \in \L^2(\Omega),
\; \sigmaboldh\in\RT_0(\taun) \cap \Hbf(\divz,\Omega).
\]
Using that~$\sigmaboldh$ is a piecewise constant field over~$\taun$,
there exists~$\Sbbh$ in $\RTbb_1(\taun) \cap \Hbb(\div,\Omega)$
with $\divbf \Sbbh = \sigmaboldh$. Direct calculations reveal
\[
\curlbf\Phi = \divbf\Obb
\qquad\qquad\text{for}\qquad\qquad
\Obb:= \left (
\begin{array}{cc}
0     & \Phi \\
-\Phi & 0
\end{array}
\right ).
\]
We deduce
\[
\Bbf
= \divbf (\Obb+\Sbbh)
\qquad\Longrightarrow\qquad
\divbf (\Abb-\Obb-\Sbbh) = \zerobf.
\]
The last identity and the tensor version of~\eqref{eq:special-split-bis} imply
\[
\Abb - \Obb - \Sbbh = \curlbb\thetabf+\Mbbh,
\qquad\qquad \text{for }
\thetabf\in\Lbf^2(\Omega),
\; \Mbbh\in \RTbb_0(\tauh) \cap \Hbb(\divbf,\Omega).
\]
Further using the symmetry of $\Abb$
and the skew-symmetry of~$\Obb$, we deduce
\[
\Abb = \symcurlbb\thetabf + \sym(\Mbbh+\Sbbh)
=: \symcurlbb\thetabf + \Qbbh,
\qquad
\Qbbh \in \RTbb_1(\taun)\cap \Hbb_{\Sbb}(\divdivbfz,\Omega).
\]
\end{proof}

\section{Proof of Proposition~\ref{proposition:tools-nc}}
\label{appendix:tools-nonconforming}

We begin with a preliminary result
stating a Korn-type inequality for the $\symcurlbb$
operator on simply connected domains.
Given a Lipschitz domain~$D$,
introduce~$\RMbf(D)$ as the space of 2D rigid body motions
and~$\RMbfperp(D)$ as the corresponding space
of rotated (by $\pi/2$) rigid body motions.

\begin{lemma} \label{lemma:preliminary-appendix-a}
Given a two dimensional simply connected,
Lipschitz domain~$D$, $\Hbf^1(D)$ coincides with~$\Hbf(\symcurlbb,D)$.
Moreover, there exists a positive constant~$\CNL$
only depending on~$D$ such that,
for all $\psibf$ in $\Hbf(\symcurlbb,\Omega)$,
we have
\begin{equation}\label{H1-bounded-symcurl}
\min_{\psibfRMbfperp\in\RMbfperp(D)}
    \SemiNorm{\psibf-\psibfRMbfperp}{1,D}
\le \CNL(D) \Norm{\symcurlbb \psibf}{D}.
\end{equation}
The constant~$\CNL(D)$ is in fact that appearing in Korn's inequality on~$D$.
\end{lemma}
\begin{proof}
Let~$\psibf:=(-\psi_2,\psi_1)$, $\psibftilde:=(\psi_1,\psi_2)$,
and~$\nablaboldS$ denote the symmetric gradient tensor.
Direct computations reveal
\[
\SemiNorm{\psibf}{1,D}
    = \SemiNorm{\psibftilde}{1,D},
\qquad\qquad\qquad
\Norm{\symcurlbb \psibf}{D}
    = \Norm{\nablaboldS \psibftilde}{D}.
\]
Korn's inequality entails
\[
\begin{split}
\min_{\psibfRMbfperp \in \RMbfperp(D)}
    \SemiNorm{\psibf-\psibfRMbfperp}{1,D}
&   = \min_{\psibfRMbf \in \RMbf(D)}
    \SemiNorm{\psibftilde-\psibfRMbf}{1,D}\\
&   \le C(D) \Norm{\sym \nablaboldS \psibftilde}{D}
    = C(D) \Norm{\symcurlbb \psibf}{D}.
\end{split}
\]
\end{proof}

We are now in a position to prove
Proposition~\ref{proposition:tools-nc}.
\begin{proof}
Consider the partition of unity $\{\phinu\}$
given by the linear Lagrangian functions
associated with the vertices in~$\Nucalh$ of~$\taun$;
the corresponding patches are denoted by $\omeganu$.
We write
\[
\psibf
=\sum_{\nu\in\Nucalh} \psibf \phinu.
\]
Define~$\PibfRMbfperpomeganu$ as the $\Hbf^1(\omeganu)$ projector
(preserving the average over~$\omeganu$)
onto the space of rotated rigid body motions~$\RMbfperp(\omeganu)$.
Given
\[
\psibfh
:=\sum_{\nu\in\Nucalh}  (\PibfRMbfperpomeganu \psibf) \phinu 
\in \VHoh,
\]
we deduce
\[
\psibf-\psibfh
= \sum_{\nu\in\Nucalh} \psibftildenu\phinu ,
\qquad\qquad\qquad
\psibftildenu
:= [\widetilde\psi_1^\nu,\widetilde\psi_2^\nu]^T
:= \psibf - \PibfRMbfperpomeganu \psibf .
\]
For any vertex~$\nu$,
standard properties of the Lagrangian partition of unity
and the Poincar\'e inequality
($\psibftildenu$ has zero average over the patch)
yield
\[
\begin{split}
\SemiNorm{\phinu \psibftildenu}{1,\omeganu}
\lesssim    \SemiNorm{\psibftildenu}{1,\omeganu}
            + \homeganu^{-1} \Norm{\psibftildenu}{\omeganu}
\lesssim \SemiNorm{\psibftildenu}{1,\omeganu}
\overset{\eqref{H1-bounded-symcurl}}{\le}
    C(\omeganu) \Norm{\symcurlbb(\psibftildenu)}{\omeganu} .
\end{split}
\]
The assertion follows summing over all patches
and observing that each element in the mesh belong to a
certain number of vertex patches,
which is bounded uniformly in terms of the shape-regularity parameter.
\end{proof}

\end{document}